\documentclass{article}

\usepackage{amsmath}
\usepackage{amssymb}
\usepackage{latexsym}
\usepackage{theorem}
\usepackage[all]{xypic}
\usepackage{graphics}

\title{\Large The $KH$-Isomorphism Conjecture and Algebraic $KK$-theory}
\author{Paul D.~Mitchener \\
University of Sheffield \\
e-mail: P.Mitchener@sheffield.ac.uk \\
Web-site: http://www.mitchener.staff.shef.ac.uk}

\newenvironment{proof}{\par \noindent{\bf Proof: }}{\hspace{\stretch{1}} $\Box$ \par \mbox{}}
\newcommand{\noproof}{\hspace{\stretch{1}} $\Box$}
\newtheorem{theorem}{Theorem}[section]
\newtheorem{proposition}[theorem]{Proposition}
\newtheorem{lemma}[theorem]{Lemma}
\newtheorem{corollary}[theorem]{Corollary}
{\theorembodyfont{\rmfamily}
\newtheorem{definition}[theorem]{Definition}

}
\newenvironment{theorem*}{\par \medskip \noindent{\bf Theorem }}{\par \mbox{}}
\newenvironment{lemma*}{\par \medskip \noindent{\bf Theorem }}{\par \mbox{}}

\newcommand{\Hom}{\mathop{Hom}}
\newcommand{\End}{\mathop{End}}

\newcommand{\Ob}{\mathop{Ob}}

\newcommand{\KK}{{\mathbb K}{\mathbb K}}
\newcommand{\N}{\mathbb {N}}
\newcommand{\R}{\mathbb {R}}

\newcommand{\K}{\mathbb {K}}
\newcommand{\Z}{\mathbb {Z}}
\newcommand{\Simp}{\mathrm{Simp}}

\newcommand{\sd}{\mathrm{sd}}
\newcommand{\ind}{\mathrm{ind}}
\newcommand{\op}{\mathrm{op}}
\newcommand{\Alg}{\mathrm{Alg}}
\newcommand{\HOM}{\mathrm{HOM}}

\newcommand{\Or}{\mathop{Or}}

\newcommand{\KH}{{\mathbb K}{\mathbb H}}

\newcommand{\EG}{\underline{E}G}

\begin{document}

\maketitle

\section*{Abstract}

In this article we prove that the $KH$-asembly map, as defined by
Bartels and L{\"u}ck, can be described in terms of the algebraic
$KK$-theory of Cortinas and Thom.  The $KK$-theory description of
the $KH$-assembly map is similar to that of the Baum-Connes assembly
map.  In some elementary cases, methods used to prove the
Baum-Connes conjecture also apply to the $KH$-isomorphism
conjecture.

\tableofcontents

\section{Introduction}

Various assembly maps, such as the Baum-Connes assembly map (see
\cite{BCH, Val}) and the Farrell-Jones assembly map (see \cite{FJ})
can be described using an abstract homotopy-theoretic framework,
developed by Davis and L{\"u}ck in \cite{DL}; the machinery is based
upon spectra and equivariant homology theories.

Recently, in \cite{BL}, Bartels and L{\"u}ck introduced the {\em
$KH$-assembly map}.  The $KH$-assembly map is constructed using the
abstract machinery, and resembles the Farrell-Jones assembly map in
algebraic $K$-theory.  However, the $KH$-assembly map uses {\em
homotopy algebraic $K$-theory}, as introduced in \cite{We}, rather
than ordinary algebraic $K$-theory.

Kasparov's bivariant $K$-theory, or $KK$-theory (see for example \cite{Black, Hig3, Ska, Val}
for accounts) is in essence a way of organizing certain maps between
the $K$-theory groups of $C^\ast$-algebras.  The Baum-Connes
assembly map is such a map, and is usually described using
$KK$-theory; some work is needed (see \cite{HP,Mitch6}) to show that the
$KK$-theory assembly map can be constructed using the
Davis-L{\"u}ck machinery.

Recently, in \cite{CoT}, Cortinas and Thom developed a version of bivariant
{\em algebraic} $K$-theory.  This version of $KK$-theory induces
maps between the homotopy algebraic $K$-theory groups of discrete
algebras.  The construction is broadly speaking similar to the
construction of $KK$-theory for locally convex algebras in
\cite{CuT}, but replaces smooth homotopies with algebraic
homotopies.

The purpose of this article is essentially an algebraic version of the converse of
\cite{Mitch6}; we show that the abstract $KH$-assembly map can be
described in terms of algebraic $KK$-theory in terms similar to that
of the Baum-Connes conjecture.  We also look at some elementary
consequences of this description, namely some elementary cases where
one can easily show that the $KH$-assembly map is an isomorphism.

It is fairly easy to write down an algebraic $KK$-theory version of
the Baum-Connes assembly map.  However, to associate this map to the
Davis-L{\"u}ck machinery, we need to write this map at the level of
spectra, and moreover to generalize algebraic $KK$-theory from
algebras to algebroids, and to equivariant algebras based on
groupoids.

Thus, before we look at assembly maps, we describe $KK$-theory
spectra for algebroids and groupoid algebras, before looking at some
basic properties.  These definitions and properties are mainly, but
not entirely, easy generalisations of those in \cite{CoT}.

\section{Algebroids}

The following definition is a slight generalization of the notion of an algebroid in \cite{Mi}.

\begin{definition}
Let $R$ be a commutative unital ring.  A {\em $R$-algebroid}, $\mathcal A$, consists of a set of {\em objects}, $\Ob ({\mathcal A})$, along
with a left $R$-module, $\Hom (a,b)_{\mathcal A}$ for each pair of
objects $a,b\in \Ob ({\mathcal A})$, such that:

\begin{itemize}

\item We have an associative $R$-bilinear {\em composition law}
$$\Hom (b,c)_{\mathcal A} \times \Hom (a,b)_{\mathcal A} \rightarrow
\Hom (a,c)_{\mathcal A}$$

\item Given an element $r\in R$ and morphisms $x\in \Hom
(a,b)_{\mathcal A}$, $y\in \Hom (b,c)_{\mathcal A}$, the equation
$r(xy) = (rx)y = x(ry)$ holds.

\end{itemize}

\end{definition}

We call an $R$-algebroid $\mathcal A$ {\em unital} if it is a
category.  Thus, in a unital $R$-algebroid we have an identity
element $1_a \in \Hom (a,a)_{\mathcal A}$ for each object $a$.

An $R$-algebra can be considered to be the same thing as an
$R$-algebroid with one object.

\begin{definition}
Let $\mathcal A$ and $\mathcal B$ be $R$-algebroids.  Then a
{\em homomorphism} $\alpha \colon {\mathcal A}\rightarrow {\mathcal
B}$ consists of maps $\alpha \colon \Ob ({\mathcal A})\rightarrow
\Ob
({\mathcal B})$ and $R$-linear maps $\alpha \colon \Hom
(a,b)_{\mathcal A}\rightarrow \Hom (\alpha (a), \alpha
(b))_{\mathcal B}$ that are compatible with the composition law.
\end{definition}

Given $R$-algebroids $\mathcal A$ and $\mathcal B$, we write $\Hom
({\mathcal A},{\mathcal B})$ to denote the set of all homomorphisms.

\begin{definition} \label{addc}
Let ${\mathcal A}$ be an $R$-algebroid.  Then we define the {\em additive
completion}, ${\mathcal A}_\oplus$, to be the $R$-algebroid in which
the objects are formal sequences of the form
$$a_1 \oplus \cdots \oplus a_n \qquad a_i \in \Ob ({\mathcal A})$$
where $n\in \N$.  Repetitions are allowed in such formal sequences.
The empty sequence is also allowed, and labelled $0$.

The $R$-module $\Hom (a_1 \oplus \cdots \oplus a_m , b_1 \oplus
\cdots \oplus b_n )_{{\mathcal A}_\oplus }$ is defined to be the set
of matrices of the form
$$\left( \begin{array}{ccc}
x_{11} & \cdots & x_{1m} \\
\vdots & \ddots & \vdots \\
x_{n1} & \cdots & x_{nm} \\
\end{array} \right)  \qquad x_{ij} \in \Hom (a_j , b_i )_{\mathcal A}$$
with element-wise addition and multiplication by elements of the
ring $R$.

The composition law is defined by matrix multiplication.
\end{definition}

Given an algebroid homomorphism $\alpha \colon {\mathcal A} \rightarrow
{\mathcal B}$, there is an induced homomorphism $\alpha_\oplus
\colon {\mathcal A}_\oplus \rightarrow {\mathcal B}_\oplus$ defined by writing
$$\alpha_\oplus (a_1 \oplus \cdots a_n ) = \alpha (a_1) \oplus \cdots \oplus \alpha (a_n) \qquad a_i \in \Ob ({\mathcal A})$$
and
$$\alpha_\oplus \left( \begin{array}{ccc}
x_{1,1} & \cdots & x_{1,m} \\
\vdots & \ddots & \vdots \\
x_{n,1} & \cdots & x_{n,m} \\
\end{array} \right) =
\left( \begin{array}{ccc}
\alpha (x_{1,1}) & \cdots & \alpha (x_{1,m}) \\
\vdots & \ddots & \vdots \\
\alpha (x_{n,1}) & \cdots & \alpha (x_{n,m}) \\
\end{array} \right)  \qquad x_{i,j} \in \Hom (a_j , b_i )$$

With such induced homomorphisms, the process of additive completion
defines a functor from the category of $R$-algebroids and
homomorphisms to itself.

We define the {\em direct sum} of two objects $a= a_1\oplus \cdots
\oplus a_m$ and $b=b_1 \oplus \cdots \oplus b_n$ in the additive
completion ${\mathcal A}_\oplus$ by writing
$$a\oplus b = a_1 \oplus \cdots \oplus a_m \oplus b_1 \oplus \cdots
\oplus b_n$$

In the unital case, the additive completion of a unital
$R$-algebroid is an additive
category;\footnote{See \cite{Mac} or \cite{We2} for relevant
definitions.} direct sums of objects can be defined as above.  The
induced functor, $\alpha_\oplus$, is an additive functor.
Analogous properties hold in the non-unital case.

Finally, given an $R$-algebroid $\mathcal A$, we would like to define an
associated $R$-algebra that carries the same information as the
additive completion
${\mathcal A}_\oplus$.  The naive way to do this is to simply form the direct
limit
$$\bigcup_{a_i \in \Ob ({\mathcal A})} \Hom (a_1 \oplus \cdots \oplus a_n,a_1 \oplus \cdots \oplus a_n)_{{\mathcal A}_\oplus}$$
with respect to the inclusions
$$\Hom (a\oplus c,a\oplus c)_{{\mathcal A}_\oplus} \rightarrow \Hom (a\oplus b\oplus c, a\oplus b\oplus c)_{{\mathcal A}_\oplus}$$
\[
\left( \begin{array}{cc}
w & x \\
y & z \\
\end{array} \right)
\mapsto \left( \begin{array}{ccc}
w & 0 & x \\
0 & 0 & 0 \\
y & 0 & z \\
\end{array} \right)
\]

Unfortunately, the above construction is not functorial.  We can,
however, replace it by an equivalent functorial construction.  This
construction is essentially the same as that used to define the
$K$-theory of $C^\ast$-categories in \cite{Jo2} or \cite{Mitch2.5}.

Let $\mathcal A$ be an $R$-algeboid.  Then we define
${\mathcal O}_{\mathcal A}$ be the category in which the set of
objects consists of all compositions of inclusions $\Hom (a\oplus
c,a\oplus c)_{{\mathcal A}_\oplus} \rightarrow \Hom (a\oplus b\oplus
c,a\oplus b\oplus c)_{{\mathcal A}_\oplus}$ of the form
\[
\left( \begin{array}{cc}
w & x \\
y & z \\
\end{array} \right)
\mapsto \left( \begin{array}{ccc}
w & 0 & x \\
0 & 0 & 0 \\
y & 0 & z \\
\end{array} \right)
\]

A morphism set between two inclusions has precisely one element if
the inclusions are composable; otherwise, it is empty.

We can define a functor, $H_{\mathcal A}$, from the category
${\mathcal O}_{\mathcal A}$ to the category of $R$-algebras by
associating the $R$-algebra $\Hom (a\oplus c,a\oplus c)_{{\mathcal
A}_\oplus}$ to the inclusion $\Hom (a\oplus c,a\oplus c)_{{\mathcal
A}_\oplus} \rightarrow \Hom (a\oplus b\oplus c,a\oplus b\oplus
c)_{{\mathcal A}_\oplus}$.  If $i$ and $j$ are composable
inclusions, then the one morphism in the set $\Hom (i,j)_{{\mathcal
O}_{\mathcal A}}$ is mapped to the inclusion $i$ itself.

\begin{definition} \label{Hfunctor}
Let $\mathcal A$ be an $R$-algebroid.  Then we define the
$R$-algebra ${\mathcal A}_H$ to be the colimit of the functor
$H_{\mathcal A}$.
\end{definition}

The following result is obvious from our constructions.

\begin{proposition} \label{pbFAC}
The assignment ${\mathcal A}_\oplus \mapsto {\mathcal A}_H$ is a
covariant functor, and we have a natural transformation
$J\colon {\mathcal A}_\oplus \rightarrow {\mathcal A}_H$.  The
natural transformation $F$ is surjective on each morphism set.

Further, given a homomorphism $\alpha \colon {\mathcal A}\rightarrow
{\mathcal B}_H$, we have a factorisation
\[ {\mathcal A} \stackrel{\alpha'}{\rightarrow } {\mathcal B}_\oplus
\stackrel{J}{\rightarrow } {\mathcal B}_H . \]
\noproof
\end{proposition}

\section{Tensor Products} \label{tensor}

\begin{definition}
Let $\mathcal A$ and $\mathcal B$ be $R$-algebroids.  Then we define
the {\em tensor product} ${\mathcal A}\otimes_R {\mathcal B}$ to be
the $R$-algebroid with the set of objects $\Ob ({\mathcal A})\times
\Ob ({\mathcal B})$.  We write the object $(a,b)$ in the form
$a\otimes b$.  The $R$-module $\Hom (a\otimes b,a'\otimes
b')_{{\mathcal A}\otimes_R {\mathcal B}}$ is the tensor product of
$R$-modules $\Hom (a,a')_{\mathcal A}\otimes_R \Hom
(b,b')_{\mathcal B}$.

The composition law is defined in the obvious way.
\end{definition}

If we view an $R$-algebra as an $R$-algebroid with just
one object, we can form the tensor product,
${\mathcal A}\otimes_R B$, of an $R$-algebroid $\mathcal A$ with an
$R$-algebra $B$.  The objects of the tensor product ${\mathcal
A}\otimes_R B$ are identified with the objects of the algebroid
$\mathcal A$.

\begin{definition}
Let $R$ be a commutative ring with an identity element.  An {\em $R$-moduloid}, $\mathcal E$,
consists of a collection of objects $\Ob ({\mathcal E})$, along with
a left $R$-module, $\Hom (a,b)_{\mathcal E}$ defined for each pair
of objects $a,b\in \Ob ({\mathcal E})$.

A {\em homomorphism}, $\phi \colon {\mathcal E} \rightarrow
{\mathcal F}$, between $R$-moduloids consists of a map $\phi \colon
\Ob ({\mathcal E}) \rightarrow \Ob ({\mathcal F})$ and a collection
of $R$-linear maps $\phi \colon \Hom (a,b)_{\mathcal E}\rightarrow
\Hom (\phi (a), \phi (b))_{\mathcal F}$.
\end{definition}

The difference between an $R$-moduloid and an $R$-algebroid is that
there is no composition law between the various $R$-bimodules $\Hom
(a,b)_{\mathcal E}$.  There is of course a forgetful functor, $F$,
from the category of $R$-algebroids and homomorphisms to the
category of $R$-moduloids and homomorphisms.

\begin{definition}
Let $\mathcal A$ be an $R$-moduloid.  Given objects $a,b\in \Ob
({\mathcal A})$, let us define
\[
\Hom (a,b)^{\otimes (k+1)}_{\mathcal A} = \bigoplus_{c_i \in \Ob
({\mathcal A})} \Hom (a, c_1)\otimes_R \Hom (c_1,c_2) \otimes_R
\cdots \otimes \Hom (c_k,b)
\]

The {\em tensor algebroid}, $T {\mathcal A}$, is the $R$-algebroid
with the same set of objects as the $R$-moduloid $\mathcal A$ where
the morphism set $\Hom (a,b)_{T{\mathcal A}}$ is the direct sum
\[
\bigoplus_{k=1}^\infty \Hom (a,b)^{\otimes k}_{\mathcal A}
\]

Here the the $R$-module $\Hom (a,b)^{\otimes 1}_{\mathcal A}$ is simply
the morphism set $\Hom (a,b)_{\mathcal A}$.  Composition of
morphisms in the tensor category is defined by concatenation of
tensors.
\end{definition}

Formation of the tensor category defines a functor, $T$, from the
category of $R$-moduloids to the category of $R$-algebroids.  We
will abuse notation slightly, and also write $T{\mathcal A}$ to
denote the tensor algebroid when $R$ is already an $R$-algebroid.
This definition of course ignores the multiplicative structure.

\begin{proposition} \label{TFadjoint}
The functor $T$ is naturally left-adjoint to the forgetful functor
$F$.
\end{proposition}

\begin{proof}
We need a natural $R$-linear bijection between the morphism sets
$\Hom
(T{\mathcal A},{\mathcal B})$ and $\Hom ({\mathcal A}, F{\mathcal
B})$ when $\mathcal A$ is an $R$-moduloid and $\mathcal B$ is an
$R$-algebroid.

Let $\alpha \colon T{\mathcal A}\rightarrow {\mathcal B}$ be a
homomorphism of $R$-algebroids.  The morphism set $\Hom
(a,b)_{T{\mathcal A}}$ is the sum
$$\bigoplus_{k=1 \atop c_i \in \Ob ({\mathcal A})}^\infty
\Hom (a, c_1)\otimes_R \Hom (c_1,c_2) \otimes_R \cdots \otimes \Hom
(c_k,b)$$

We have an induced homomorphism of $R$-moduloids, $G({\alpha
})\colon {\mathcal A}\rightarrow F{\mathcal B}$, defined to be
$\alpha$ on the set of objects, and by the restriction $G({\alpha})
= \alpha |_{\Hom (a,b)_{\mathcal A}}$ on morphism sets.

Conversely, given an $R$-moduloid homomorphism $\beta \colon
{\mathcal A}\rightarrow F{\mathcal B}$, we have an $R$-algebroid
homomorphism $H({\beta }) \colon T{\mathcal A}\rightarrow {\mathcal
B}$, defined to be $\beta$ on the set of objects, and by the formula
$$H({\beta })(x_1 \otimes \cdots \otimes x_k) = \beta (x_1)\cdots \beta (x_k)$$
for each morphism of the form $x_1 \otimes \cdots \otimes x_k$ in
the tensor category.

It is easy to check that the maps $G$ and $H$ are $R$-linear, natural, and mutually inverse.
\end{proof}

Now, let $\mathcal A$ be an $R$-algebroid.  Then there is a
canonical homomorphism {\em of $R$-moduloids} $\sigma \colon
{\mathcal A}\rightarrow T{\mathcal A}$ defined by mapping each
morphism set of the category $\mathcal A$ onto the first summand.

\begin{proposition}
Let $\alpha \colon {\mathcal A}\rightarrow {\mathcal B}$ be a
homomorphism of $R$-algebroids.  Then there is a unique
$R$-algebroid homomorphism $\varphi \colon T{\mathcal A}\rightarrow
{\mathcal B}$ such that $\alpha = \varphi \circ \sigma$.
\end{proposition}

\begin{proof}
We can define the required homomorphism $\varphi \colon
T{\mathcal A}\rightarrow {\mathcal B}$ by writing
$\varphi (a) = \alpha (a)$ for each object $a\in \Ob ({\mathcal
A})$, and
\[
\varphi (x_1 \otimes \cdots \otimes x_n) = \alpha (x_n)\ldots \alpha
(x_1)
\]
for morphisms $x_i \in \Hom (c_i,c_{i+1})_{{\mathcal A}}$.  It is
easy to see that
$\varphi$ is the unique $R$-algebroid homomorphism with the property
that $\alpha =\varphi \circ \sigma$.
\end{proof}

We have a natural homomorphism $\pi \colon T{\mathcal A}\rightarrow
{\mathcal A}$ defined to be the identity on the set of objects, and
by the formula
\[
\varphi (x_1 \otimes \cdots \otimes x_n) = x_n\ldots x_1
\]
for morphisms $x_i \in \Hom (c_i,c_{i+1})$.  It follows that there
is an $R$-algebroid $J{\mathcal A}$ with the same objects as the
category $\mathcal A$, and morphism sets
\[
\Hom (a,b)_{J{\mathcal A}} = \ker (\pi \colon \Hom (a,b)_{T{\mathcal
A}} \rightarrow \Hom (a,b)_{\mathcal A})
\]

\begin{definition}
A sequence of $R$-algebroids and homomorphisms
$$0\rightarrow {\mathcal I}\stackrel{i}{\rightarrow} {\mathcal
E}\stackrel{j}{\rightarrow} {\mathcal B}\rightarrow 0$$ is called a
{\em short exact sequence} if the $R$-algeboids $\mathcal I$,
$\mathcal E$, and $\mathcal B$ all have the same sets of objects,
and for all such object $a$ and
$b$ our sequence restricts to a short exact sequence of abelian
groups:
$$0\rightarrow \Hom (a,b)_{\mathcal I} \rightarrow \Hom
(a,b)_{\mathcal A} \rightarrow \Hom (a,b)_{\mathcal E}\rightarrow
0$$

We call the above short exact sequence {\em $F$-split} if there is a
homomorphism of $R$-moduloids (but not necessarily $R$-algebroids),
$s\colon {\mathcal B}\rightarrow {\mathcal E}$ such that $j\circ s
=1_{\mathcal C}$.

We call the homomorphism $s$ an {\em $F$-splitting} of the short
exact sequence.
\end{definition}

Our definitions make it clear that, for an $R$-algeboid $\mathcal
A$, the tensor algebroid fits into a natural short exact sequence
\[
0\rightarrow J{\mathcal A} \hookrightarrow T{\mathcal A}
\stackrel{\pi}{\rightarrow} {\mathcal A}\rightarrow 0
\]
with natural $F$-splitting $\sigma \colon {\mathcal A}\rightarrow
T{\mathcal A}$.

\begin{definition}
The above short exact sequence is called the {\em universal
extension} of $\mathcal A$.
\end{definition}

The following result is easy to check.

\begin{proposition} \label{jtens}
Let $\mathcal A$ and $\mathcal C$ be $R$-algebroids.  Then we have a
natural $F$-split short exact sequence
$$0 \rightarrow (J{\mathcal A})\otimes_R {\mathcal C} \rightarrow (T{\mathcal A})\otimes_R {\mathcal C} \rightarrow {\mathcal A}\otimes_R {\mathcal C} \rightarrow 0$$
\noproof
\end{proposition}

\begin{theorem} \label{ufsplit}
Let
$$0\rightarrow {\mathcal I}\stackrel{i}{\rightarrow} {\mathcal E} \stackrel{j}{\rightarrow} {\mathcal A}\rightarrow 0$$
be an $F$-split short exact sequence.  Then we have a natural
homomorphism $\gamma \colon J{\mathcal A}\rightarrow {\mathcal I}$
fitting into a commutative diagram
$$\begin{array}{ccccc}
J{\mathcal A} & \rightarrow & T{\mathcal A} & \rightarrow & {\mathcal A} \\
\downarrow & & \downarrow & & \| \\
{\mathcal I} & \rightarrow & {\mathcal E} & \rightarrow & {\mathcal A} \\
\end{array}$$

If our $F$-splitting is a homomorphism of $R$-algebroids,
then the homomorphism $\gamma$ is the zero map.
\end{theorem}

\begin{proof}
Since the exact sequence we are looking at is $F$-split, there is an
$R$-moduloid homomorphism $s\colon {\mathcal A}\rightarrow {\mathcal
E}$ such that $j\circ s =1_{\mathcal A}$.  By proposition
\ref{TFadjoint}, there is a natural $R$-algebroid homomorphism
$H(s)\colon T{\mathcal A}\rightarrow {\mathcal E}$ fitting into the
above diagram.

The homomorphism $\gamma$ is defined by restriction of the
homomorphism $H(s)$.

Now, suppose that the above $R$-moduloid homomorphism $s$ is
actually an algebroid homomorphism.  Then, by exactness, we have a
commutative diagram
$$\begin{array}{ccccc}
J{\mathcal A} & \rightarrow & T{\mathcal A} & \stackrel{\pi}{\rightarrow} & {\mathcal A} \\
\downarrow & & \downarrow & & \| \\
{\mathcal I} & \rightarrow & {\mathcal E} & \rightarrow & {\mathcal A} \\
\end{array}$$
where the vertical homomorphism on the left is zero, and central
vertical homomorphism is the composition $s\circ \pi$.
\end{proof}

\begin{definition}
We call the above homomorphism $\gamma$ the {\em classifying map} of
the diagram
$$\begin{array}{ccccc}
J{\mathcal A} & \rightarrow & T{\mathcal A} & \rightarrow & {\mathcal A} \\
\downarrow & & \downarrow & & \| \\
{\mathcal I} & \rightarrow & {\mathcal E} & \rightarrow & {\mathcal A} \\
\end{array}$$
\end{definition}

\section{Algebraic Homotopy and Simplicial Enrichment} \label{alghom}

For each natural number, $n\in \N$, write
\[ \Z^{\Delta^n} = \frac{\Z [t_0,\ldots ,t_n]}{\langle 1-\sum_{i=0}^n
t_i \rangle } \]

In particular, we have $\Z^{\Delta^0} = \Z$ and $\Z^{\Delta^1} = \Z
[t]$.  The sequence of rings $(\Z^{\Delta^n})$ combines with the
obvious face and degeneracy maps to form a simplicial ring
$\Z^\Delta$.  We refer to \cite{CoT} for details.

Given an $R$-algebroid $\mathcal A$, we can also consider $\mathcal
A$ to be a $\Z$-algebroid, and so form tensor products ${\mathcal
A}\otimes_\Z \Z^{\Delta^n}$, and the simplicial $R$-algebroid
${\mathcal A}^\Delta$.

In the case $n=1$, there are obvious homomorphisms $e_i \colon {\mathcal A}^{\Delta^1}\rightarrow
{\mathcal A}$ defined by the evaluation of a polynomial at a point
$i\in \Z$.

\begin{definition}
Let $f_0,f_1 \colon {\mathcal A}\rightarrow {\mathcal B}$ be
homomorphisms of $R$-algebroids.  An {\em elementary homotopy}
between $f_0$ and $f_1$ is a homomorphism $h\colon {\mathcal
A}\rightarrow {\mathcal B}^{\Delta^1}$ such that $e_0 \circ h =f_0$
and $e_1 \circ h = f_1$.

We call $f_0$ and $f_1$ {\em algebraically homotopic} if they are
linked by a chain of elementary homotopies.  We write
$[{\mathcal A},{\mathcal B}]$ to denote the set of algebraic homotopy
classes of homomorphisms from $\mathcal A$ to $\mathcal B$.
\end{definition}

Let us view the geometric $n$-simplex $\Delta^n$ as a simplicial set.
Given a simplicial set $X$, recall (see for example \cite{GJ}) that
a {\em simplex} of $X$ is a simplicial map $f\colon \Delta^n
\rightarrow X$.  We define the {\em simplex category of $X$},
$\Delta \downarrow X$, to be the category in which the objects are
the simplices of $X$, and the morphisms are commutative diagrams
$$\begin{array}{ccc}
\Delta^n & \rightarrow & X \\
\downarrow & & \| \\
\Delta^m & \rightarrow & X \\
\end{array}$$

The simplex category has the feature that the simplicial set $X$ is
naturally isomorphic to the direct limit
$$\lim_{\Delta^n \rightarrow X} \Delta^n$$
taken over the simplex category.

\begin{definition}
Let $X$ be a simplicial set.  Let $\mathcal A$ be an $R$-algebroid.
Then we define ${\mathcal A}^X$ to be the direct limit
$${\mathcal A}^X = \lim_{\Delta^n \rightarrow X} {\mathcal A}^{\Delta^n}$$
taken over the simplex category.
\end{definition}

The assignment ${\mathcal A}^- \colon X\mapsto {\mathcal A}^X$ is a
contravariant functor.  Let $X$ be a simplicial set with a
basepoint, $+$ (in the sense of simplicial sets; see again
\cite{GJ}).
Then there is a functorially induced map
${\mathcal
A}^X \rightarrow {\mathcal A}={\mathcal A}^+$ arising from the inclusion
$+\hookrightarrow X$.  We define
$${\mathcal A}^{X,+} = \ker ({\mathcal A}^X\rightarrow {\mathcal A})$$

Given pointed simplicial sets $X$ and $Y$, it is easy to check that the
$R$-algebroids $({\mathcal A}^X)^Y$ and ${\mathcal A}^{X\wedge Y}$
are naturally isomorphic.  The product $X\wedge Y$ is of course the
smash product of pointed simplicial sets.  In particular, if $S^m$
and $S^n$ are simplicial spheres, then the $R$-algebroids
$({\mathcal A}^{S^m})^{S^n}$ and ${\mathcal A}^{S^{m+n}}$ are
naturally isomorphic.

Recall from proposition \ref{pbFAC} that there is a functor ${\mathcal A}\mapsto {\mathcal A}_H$
from the category of $R$-algebroids to the category of $R$-algebras.
Naturality of the relevant constructions immediately gives us the
following result.

\begin{proposition}
Let $\mathcal A$ be an $R$-algebroid, and let $X$ be a simplicial
set.  Then the $R$-algebras $({\mathcal A}_H)^X$ and
$({\mathcal
A}^X)_H$ are naturally isomorphic.
\noproof
\end{proposition}

Because of the above proposition, we do not distinguish between the
$R$-algebras $({\mathcal A}_H)^X$ and $({\mathcal A}^X)_H$, and
simplify our notation by writing ${\mathcal A}_H^X$ in both cases.

Now, let $\Simp$ denote the category of simplicial sets.  Let
$\Alg_R$ denote the category of $R$-algebroids.  Then the
contravariant functor ${\mathcal A}^-$ can be written as a covariant
functor ${\mathcal A}^-  \colon \Simp^\op \rightarrow \Alg_R$.

Let us call a directed object in the category of $\mathcal
R$-algebroids and homomorphisms a {\em directed $\mathcal
R$-algebroid}.  If we write $\Alg_R^\ind$ to denote the category of directed $R$-algebroids,
then the above functor has an extension
$${\mathcal A}^- \colon \Simp^\op \rightarrow \Alg_R^\ind$$

Given a simplicial set $X$, we can form its subdivision $\sd
(X)$; there is a natural simplicial map $h\colon \sd (X) \rightarrow
X$.  The process of repeated subdivision yields a sequence of simplicial sets, $\sd^\bullet X$:
$$\sd^0 X \stackrel{h_X}{\leftarrow} \sd^1 X \stackrel{h_{\sd (X)}}{\leftarrow} \sd^2 X \leftarrow \cdots$$

\begin{definition}
Let $\mathcal A$ and $\mathcal B$ be $R$-algebroids.  Then we define
$\HOM ({\mathcal A},{\mathcal B})$ to be the simplicial set defined
by writing
$$\HOM ({\mathcal A},{\mathcal B})[n] = \lim_{\rightarrow \atop k} \Hom ({\mathcal A}, {\mathcal B}^{\sd^k \Delta^n})$$

The face and degeneracy maps are those inherited from the simplicial
$R$-algebroid ${\mathcal B}^\Delta$.
\end{definition}

The following result will by used later on in this article to
formulate algebraic $KK$-theory in terms of spectra.  It is proven
in exactly the same way as theorem 3.3.2 in \cite{CoT}.

\begin{theorem} \label{HOM}
Let $\mathcal A$ be a simplicial $R$-algebroid, and let $\mathcal B$
be a directed $R$-algebroid.  Then
$$[{\mathcal A},{\mathcal B}^{S^n,+}] = \pi_n \HOM ({\mathcal A},{\mathcal B})$$
\noproof
\end{theorem}

Of course, in the above theorem, we use $S^n$ to denote the
simplicial $n$-sphere, and we are using simplicial homotopy groups;
 see for instance \cite{GJ} or some other standard reference on simplicial homotopy
theory for further details.

\section{Path Extensions} \label{path}

Let $\mathcal A$ be an $R$-algebroid.  Let ${\mathcal A}\oplus
{\mathcal A}$ be the $R$-algebroid with the same objects as
$\mathcal A$, with morphism sets
\[ \Hom (a,b)_{{\mathcal A}\oplus {\mathcal A}} = \Hom
(a,b)_{\mathcal A}\oplus \Hom (a,b)_{\mathcal A} \]

Write $\Omega {\mathcal A} = {\mathcal A}^{S^1 ,+}$.  Then we obtain
a short exact sequence
$$0\rightarrow \Omega {\mathcal A} \rightarrow {\mathcal A}^{\Delta^1} \stackrel{(e_0,e_1)}{\rightarrow} {\mathcal A}\oplus {\mathcal A}\rightarrow 0$$
with an $F$-splitting $s\colon {\mathcal A}\oplus {\mathcal A}
\rightarrow {\mathcal A}^{\Delta^1}$ given by the formula
$$s(x,y) = (1-t) x + ty$$

If we write $P {\mathcal A} = {\mathcal A}^{\Delta^1 ,+}$, we have
a commutative diagram
$$\begin{array}{ccccc}
\Omega {\mathcal A} & \rightarrow & P{\mathcal A} & \stackrel{e_1}{\rightarrow} & {\mathcal A} \\
\| & & \downarrow & & \downarrow \\
\Omega {\mathcal A} & \rightarrow & {\mathcal A}^{\Delta^1} & \stackrel{(e_0,e_1)}{\rightarrow} & {\mathcal A}\oplus {\mathcal A} \\
& & \downarrow & & \downarrow \\
& & {\mathcal A} & = & {\mathcal A} \\
\end{array}$$
where the homomorphism ${\mathcal A}\rightarrow {\mathcal A}\oplus
{\mathcal A}$ is the inclusion of the second factor, the
homomorphism
${\mathcal A}\oplus {\mathcal A}\rightarrow {\mathcal A}$ is projection onto the first factor,
and the map
${\mathcal A}^{\Delta^1}\rightarrow {\mathcal A}$ is the evaluation map $e_0$.
Further, the complete rows and columns are short exact sequences.

The column on the right has a natural splitting $(1,0)\colon
{\mathcal A}\rightarrow {\mathcal A}\oplus {\mathcal A}$.  The
second row is the path extension, which has a natural $F$-splitting
as we observed above.  By a diagram chase, it follows that the top
row and middle column have natural $F$-splittings.

Hence, by theorem \ref{ufsplit}, there is a natural
map $\rho \colon J{\mathcal A} \rightarrow \Omega {\mathcal A}$.

\begin{definition}
Let $f\colon {\mathcal A}\rightarrow {\mathcal B}$ be a homomorphism of $R$-algebroids.  Then the {\em path algebroid} of $f$ is the unique
$R$-algebra $P{\mathcal B}\oplus_{\mathcal B}{\mathcal A}$ fitting into a commutative
diagram
$$\begin{array}{ccccc}
\Omega {\mathcal B} & \rightarrow & P{\mathcal B}\oplus_{\mathcal B}{\mathcal A} & \rightarrow & {\mathcal A} \\
\| & & \downarrow & & \downarrow \\
\Omega {\mathcal B} & \rightarrow & P{\mathcal B} & \rightarrow & {B} \\
\end{array}$$
where the bottom row is the path extension, and the upper row is a
short exact sequence.
\end{definition}

It is straightforward to check that the path algebroid is
well-defined, and that the upper row in the above diagram has a
natural $F$-splitting.  There is therefore a natural map
$\eta (f)\colon J{{\mathcal A}}\rightarrow \Omega {{\mathcal B}}$.

We can compose with the map $i\colon \Omega {\mathcal A}\rightarrow
{\mathcal A}^{S^1}$ to obtain a map $\eta (f) \colon J{\mathcal A}
\rightarrow {\mathcal B}^{S^1}$.

\begin{definition}
We call the above homomorphism $\eta (f)$ the {\em classifying map}
of the homomorphism $f$.
\end{definition}

\section{The $KK$-theory spectrum}

Consider a homomorphism $\alpha \colon {\mathcal A}\rightarrow
{\mathcal B}_H$.  Then we have a classifying map $\eta (\alpha )
\colon
J{\mathcal A}\rightarrow {\mathcal B}_H^{S^1}$.  Given a a natural
number $m\in {\mathbb N}$, we define the $R$-algebroid $J^m
{\mathcal
A}$ iteratively, by writing
$$J^0 {\mathcal A} = {\mathcal A} \qquad J^{k+1}{\mathcal A} = J(J^k
{\mathcal A})$$

Given a homomorphism $\alpha \colon J^{2n}{\mathcal A} \rightarrow
({\mathcal B}_H^{S^n})^{\sd^\bullet \Delta^n}$, we have a structure map
$$\eta (\alpha ) \colon J^{2n+1} \rightarrow ({\mathcal
B}_H^{S^{n+1}})^{\sd^\bullet \Delta^n}$$

Applying the classifying map construction again to the above, we see
that we have a simplicial map
$$\epsilon \colon \HOM (J^{2n}{\mathcal A}, {\mathcal B}_H^{S^n})
\rightarrow \HOM (J^{2n+2}{\mathcal A},{\mathcal B}_H^{S^{n+2}})
\cong \Omega \HOM (J^{2n+2}{\mathcal A},{\mathcal B}_H^{S^{n+1}})$$

\begin{definition}
We define $\KK({\mathcal A},{\mathcal B})$ to be the spectrum with
sequence of spaces $\HOM (J^{2n} {\mathcal A},{\mathcal
B}_H^{S^n})$, with structure maps defined as above.
\end{definition}

Note that, by proposition \ref{pbFAC}, elements of the space $\HOM (J^{2n} {\mathcal A},{\mathcal
B}_H^{S^n})$ are defined by homomorphisms $\alpha \colon J^{2n}
{\mathcal A} \rightarrow ({\mathcal B}_\oplus^{S^n})^{\sd^k
\Delta^n}$.

Recall from \cite{HSS} that a {\em symmetric spectrum} is a
spectrum, $\mathbb E$, equipped with actions of symmetric groups
$\Sigma_n\times E_n\rightarrow E_n$ that commute with the relevant
structure maps.  The extra structure means that the {\em smash
product}, ${\mathbb E}\wedge {\mathbb F}$ of symmetric spectra
$\mathbb E$ and $\mathbb F$ can be defined.

\begin{proposition}
The spectrum $\KK ({\mathcal A},{\mathcal B})$ is a symmetric
spectrum.
\end{proposition}

\begin{proof}
There is a canonical action of the permutation group $\Sigma_n$ on the
simplicial sphere $S^n \cong S^1 \wedge \cdots \wedge S^1$ defined
by permuting the order of the smash product of simplicial circles,
and therefore on the space
$\HOM (J^{2n} {\mathcal A},{\mathcal B}_\oplus^{S^n})$.

By construction, the iterated structure map $\epsilon^k \colon \HOM (J^{2n} {\mathcal
A},{\mathcal B}_\oplus^{S^n})\rightarrow \Omega \HOM (J^{2n+2k} {\mathcal A},{\mathcal
B}_\oplus^{S^{n+k}})$ is $\Sigma_n \times \Sigma_k$-equivariant, and so we have a
symmetric spectrum as required.
\end{proof}

\begin{proposition}
Let $\mathcal B$ be an $R$-algebroid, and let $k$ and $l$ be natural
numbers.  Then there is a natural homomorphism $s\colon
J^k({\mathcal B}^{S^l}) \rightarrow (J^l{\mathcal B})^{S^k}$.
\end{proposition}

\begin{proof}
The classifying map of the $F$-split exact sequence
$$0 \rightarrow (J{\mathcal A})^{S^1} \rightarrow (T {\mathcal A})^{S^1} \rightarrow {\mathcal A}^{S^1} \rightarrow 0$$
is a natural homomorphism $J({\mathcal A}^{S^1})\rightarrow
(J{\mathcal A})^{S^1}$.  The homomorphism $s$ is defined by
iterating the above construction.
\end{proof}

\begin{definition} \label{Shprodalg}
Let $\mathcal A$, $\mathcal B$, and $\mathcal C$ be $R$-algebroids.
Consider homomorphisms $\alpha \colon J^{2m} {\mathcal A}\rightarrow
{\mathcal B}^{S^m}_\oplus$ and $\beta \colon J^{2n} {\mathcal
B}\rightarrow {\mathcal C}^{S^n}_\oplus $.  Then we define the
product $\alpha \sharp \beta$ to be the composition
\[
J^{2m+2n}{\mathcal A}\stackrel{J^{2n} \alpha}{\rightarrow}
J^{2n}({\mathcal B}^{S^m}_\oplus ) \stackrel{s}{\rightarrow}
(J^{2n}{\mathcal B}_\oplus )^{S^m}
\stackrel{\beta_\oplus}{\rightarrow} {\mathcal C}_\oplus^{S^{m+n}}
\]
\end{definition}

We can of course also define the above in the case of directed
$R$-algebroids.  This extension of the above definition is needed in
the following theorem.

\begin{theorem} \label{KKprodalg}
Let $\mathcal A$, $\mathcal B$, and $\mathcal C$ be $R$-algebroids.
Then there is a natural map of spectra
\[
\KK({\mathcal A},{\mathcal B}) \wedge \KK({\mathcal B},{\mathcal C})
\rightarrow \KK({\mathcal A},{\mathcal C})
\]
defined by the formula
\[
\alpha \colon
J^{2m} {\mathcal A}\rightarrow ({\mathcal B}^{S^m}_\oplus)^{\sd^\bullet
\Delta^k}
\qquad
\alpha \colon
J^{2n} {\mathcal B}\rightarrow ({\mathcal C}^{S^n}_\oplus)^{\sd^\bullet
\Delta^l}
\]

Further, the above product is associative in the obvious sense.
Given homomorphisms $\alpha \colon {\mathcal A}\rightarrow {\mathcal
B}$ and $\beta \colon {\mathcal B}\rightarrow {\mathcal C}$, we have
the formula $\alpha \sharp \beta = \beta \circ \alpha$.
\end{theorem}

\begin{proof}
For convenience, we will simply consider homomorphisms of the form
$\alpha \colon J^{2m} {\mathcal A}\rightarrow {\mathcal B}^{S^m}_\oplus$
and $\beta \colon J^{2n} {\mathcal B}\rightarrow {\mathcal
C}^{S^n}_\oplus$.

Our construction gives us a natural continuous $S_m \times
S_n$-equivariant map $\HOM (J^m{\mathcal A},{\mathcal B}_H^{S^m})
\wedge \HOM (J^n {\mathcal B}, {\mathcal C}_H^{S^n}) \rightarrow
\HOM (J^{m+n}{\mathcal A},{\mathcal C}_\oplus^{S^{m+n}})$. Compatibility
with the structure maps follows since naturality of the classifying
map construction gives us a commutative diagram
\[
\begin{array}{ccccccc}
J^{2m+2n+2}{\mathcal A} & \stackrel{J^{2n+2}(\alpha )}{\rightarrow}
& J^{2n+2} ({\mathcal B}^{S^m}_\oplus )
& \stackrel{s}{\rightarrow} & (J^{2n+2} {\mathcal B})^{S^n}_\oplus & \stackrel{\eta^2 (\beta_\oplus )}{\rightarrow} & \Sigma^{m+n+2} {\mathcal C}_\oplus^{S^{m+n+2}} \\
\| & & \downarrow & & \downarrow & & \| \\
J^{2m+2n+2}{\mathcal A} & \stackrel{J^{2n} \eta^2(\alpha )}{\rightarrow} & J^{2n}({\mathcal B}^{S^{m+2}}_\oplus ) & \stackrel{s}{\rightarrow} & (J^{2n} {\mathcal B})^{S^{m+2}}_\oplus & \stackrel{\beta_\oplus }{\rightarrow} & {\mathcal C}^{S^{m+n+2}}_\oplus \\
\end{array}
\]

We now need to check the statement concerning associativity.
Consider homomorphisms
\[
\alpha \colon J^{2m}{\mathcal A}\rightarrow {\mathcal
B}^{S^m}_\oplus \qquad \beta \colon J^{2n}{\mathcal B}\rightarrow
{\mathcal C}^{S^n}_\oplus \qquad \gamma \colon J^{2p}{\mathcal
C}\rightarrow {\mathcal D}^{S^p}_\oplus
\]

Then we have a commutative diagram
\[
\begin{array}{ccc}
J^{2m+2n+2p}{\mathcal A} & = & J^{2m+2n+2p}{\mathcal A} \\
\downarrow & & \downarrow \\
J^{2n+2p}({\mathcal B}^{S^m}_\oplus ) & = & J^{2n+2p}({\mathcal B}^{S^m}_\oplus) \\
\downarrow & & \downarrow \\
J^{2p}(J^{2n}{\mathcal B}^{S^m}_\oplus ) & \stackrel{s}{\rightarrow} & (J^{2n+2p} {\mathcal B})^{S^m})_\oplus \\
\downarrow & & \downarrow \\
J^{2p} ({\mathcal C}^{S^{m+n}}_\oplus ) & \stackrel{s}{\rightarrow}
& J^{2p} (({\mathcal C}^{S^m})^{S^n}_\oplus ) \\
\downarrow & & \downarrow \\
(J^{2p} {\mathcal C})^{S^{m+n}}_\oplus  & = & J^{2p}
({\mathcal C}^{S^{m+n}}_\oplus ) \\
\downarrow & & \downarrow \\
{\mathcal D}^{S^{m+n+p}}_\oplus & = & \Sigma^{m+n+p} {\mathcal D}^{S^{m+n+p}}_\oplus \\
\end{array}
\]

But the column on the left is the product $(\alpha \sharp \beta
)\sharp \gamma$ and the column on the right is the product $\alpha
\sharp (\beta \sharp \gamma)$ so associativity of the product
follows.
\end{proof}

\begin{proposition} \label{dmapalg}
Let $\mathcal A$, $\mathcal B$, and $\mathcal C$ be $R$-algebroids.
Then there is a map $\Delta \colon \KK ({\mathcal A},{\mathcal
B})\rightarrow \KK ({\mathcal A}\otimes_R {\mathcal C},{\mathcal
B}\otimes_R {\mathcal C})$.  This map is compatible with the product
in the sense that we have a commutative diagram
\[
\begin{array}{ccc}
\KK ({\mathcal A},{\mathcal B})\wedge \KK ({\mathcal B},{\mathcal C}) & \rightarrow & \KK ({\mathcal A},{\mathcal C}) \\
\downarrow & & \downarrow \\
\KK ({\mathcal A}\otimes_R {\mathcal D},{\mathcal B}\otimes_R {\mathcal D})\wedge \KK ({\mathcal B}\otimes_R {\mathcal D},{\mathcal C}\otimes_R {\mathcal D}) & \rightarrow & \KK ({\mathcal A}\otimes_R {\mathcal D},{\mathcal C}\otimes_R {\mathcal D}) \\
\end{array}
\]
where the horizontal maps are defined by the product, and the
vertical maps are copies of the map $\Delta$.
\end{proposition}

\begin{proof}
Let $\alpha \colon J^{2n}{\mathcal A} \rightarrow {\mathcal
B}^{S^n}_\oplus$ be a homomorphism.  Then we have a naturally
induced homomorphism $\alpha \otimes 1 \colon (J^{2n}{\mathcal
A})\otimes_R {\mathcal C}\rightarrow {\mathcal B}^{S^n}_\oplus
\otimes_R {\mathcal C})$.

There is an obvious natural homomorphism $\beta \colon {\mathcal
B}^{S^n}_\oplus \otimes_R {\mathcal C} \rightarrow ({\mathcal
B}\otimes_R {\mathcal C})^{S^n}_\oplus$.  By proposition
\ref{jtens}, we have an $F$-split short exact sequence
$$0 \rightarrow (J{\mathcal A})\otimes_R {\mathcal C} \rightarrow (T{\mathcal A})\otimes_R {\mathcal C} \rightarrow {\mathcal A}\otimes_R {\mathcal C} \rightarrow 0$$

We thus obtain a natural homomorphism $\gamma \colon J({\mathcal
A}\otimes_R {\mathcal C})\rightarrow (J{\mathcal A})\otimes_R
{\mathcal C}$ as the classifying map of the diagram
\[
\begin{array}{ccccccccc}
& & & & & & {\mathcal A}\otimes {\mathcal C} \\
& & & & & & \| \\
0 & \rightarrow & (J{\mathcal A})\otimes {\mathcal C} & \rightarrow & (T{\mathcal A})\otimes {\mathcal C} & \rightarrow & {\mathcal A}\otimes {\mathcal C} & \rightarrow & 0 \\
\end{array}
\]

We now define the map $\Delta$ by writing $\Delta (\alpha ) = \beta
\circ (\alpha \otimes 1)\circ \gamma^n$.  The relevant naturality
properties are easy to check.
\end{proof}

The following result follows directly from the relevant definitions
in the category of symmetric spectra (see \cite{HSS}) along with the
above proposition and theorem.

\begin{theorem}
Let $\mathcal A$ be an $R$-algebroid.  Then the spectrum $\KK
({\mathcal A},{\mathcal A})$ is a symmetric ring spectrum.

Let $B$ be another $R$-algebroid.  Then the spectrum $\KK ({\mathcal
A},{\mathcal B})$ is a symmetric $\KK (R ,R )$-module spectrum.
\noproof
\end{theorem}

\section{$KK$-theory groups} \label{KKg}

Note that, by theorem \ref{HOM},
$$\pi_0 \KK ({\mathcal A},{\mathcal B}) = \lim_{\rightarrow \atop n} [J^{2n}{\mathcal A}, {\mathcal B}_H^{S^{2n}}]$$
and when $B$ is an $R$-algebra, we have a natural isomorphism
${B}_H\cong B\otimes_R M_\infty (R)$, where $M_\infty
(R)$ denotes the $R$-algebra of infinite matrices, indexed by
$\mathbb N$, with entries in the ring $R$.

Thus, if we define
$$KK ({\mathcal A},{\mathcal B}) = \pi_0 \KK ({\mathcal A},{\mathcal
B})$$
then in the case where $\mathcal A$ and $\mathcal B$ are
$R$-algebras, we recover the definition of bivariant algebraic $KK$-theory in
\cite{CoT}.

\begin{definition}
Let $\mathcal A$ and $\mathcal B$ be $R$-algebroids.  Then we can
define the $KK$-theory groups
$$KK_p ({\mathcal A},{\mathcal B}) = \pi_p \KK ({\mathcal A},{\mathcal B})$$
\end{definition}

By theorem \ref{HOM}, we have natural isomorphisms
$$[{\mathcal A},{\mathcal B}^{S^n,+}] \cong \pi_n \HOM ({\mathcal A},{\mathcal B})$$

Hence, by definition of the $KK$-theory spectrum, we know that
$$KK_p ({\mathcal A},{\mathcal B}) \cong \lim_{\rightarrow \atop n} [J^{n}{\mathcal A}, {\mathcal B}_\oplus^{S^{n+p}}]$$

In section \ref{path}, that we defined an $R$-algebroid $\Omega
{\mathcal B} = {\mathcal B}^{S^1 ,+}$.  Iterating this construction,
we see that $\Omega^n {\mathcal B} = {\mathcal B}^{S^n , +}$.
Hence, when $p\geq 0$ we can write
$$KK_p ({\mathcal A},{\mathcal B}) = KK_0 ({\mathcal A},\Omega^p {\mathcal B})$$

Similarly, we can write
$$KK_{-p} ({\mathcal A},{\mathcal B}) = KK_0 (J^p{\mathcal A},{\mathcal B})$$

By definition, elements of the group $KK_0 ({\mathcal A},{\mathcal
B})$ arise from homomorphisms $\alpha \colon J^{2n} {\mathcal
A}\rightarrow {\mathcal B}^{S^n}_\oplus$.  Given two such
homomorphisms $\alpha , \beta \colon J^{2n} {\mathcal A}\rightarrow
{\mathcal B}^{S^n}_\oplus$, we can define the sum, $\alpha \oplus
\beta \colon J^{2n} {\mathcal A}\rightarrow {\mathcal
B}^{S^n}_\oplus$, by writing
$$(\alpha \oplus \beta )(a) = \alpha (a)\oplus \beta (a)$$
for each object $a\in \Ob (J^n{\mathcal A})$, and
$$(\alpha \oplus \beta )(x) = \left( \begin{array}{cc}
\alpha (x) & 0 \\
0 & \beta (x)
\end{array} \right)$$
for each morphism $x$.

The following result is obvious from the construction of
$KK$-theory.

\begin{proposition}
Let $\alpha \colon J^{2n} {\mathcal A}\rightarrow {\mathcal
B}^{S^n}_\oplus$ be a homomorphism.  Let $[\alpha ]$ be the
equivalence class defined by the following conditions.

\begin{itemize}

\item Let $\alpha ,\beta \colon J^{2n} {\mathcal A}\rightarrow {\mathcal B}^{S^n}_\oplus$ be algebraically homotopic.  Then $[\alpha ] = [\beta ]$.

\item Let $\alpha \colon J^{2n} {\mathcal A}\rightarrow {\mathcal B}^{S^n}_\oplus$ be a homomorphism, and let $0\colon {\mathcal A}\rightarrow {\mathcal B}^{S^n}_\oplus$ be the homomorphism that is zero for each morphism in the category $J^n \mathcal A$.  Then $[\alpha ] = [\alpha \oplus 0 ]$.

\item Let $\alpha \colon J^{2n} {\mathcal A}\rightarrow {\mathcal B}^{S^n}_\oplus$, and let $\eta (\alpha )\colon J^{2n+2}{\mathcal A}\rightarrow {\mathcal B}^{S^{n+1}}_\oplus$ be the corresponding classifying map.  Then $[\eta (\alpha )] = [\alpha ]$.

\end{itemize}

Then the group $KK_0 ({\mathcal A},{\mathcal B})$ is the set of
equivalence classes of homomorphisms $\alpha \colon J^n{\mathcal
A}\rightarrow {\mathcal B}^{S^n}_\oplus$.  The group operation is
defined by the formula $[\alpha \oplus \beta ] = [\alpha ]+[\beta
]$. \noproof
\end{proposition}

At the level of groups, the product is an associative map
$$KK_p ({\mathcal A},{\mathcal B})\otimes KK_q ({\mathcal B},{\mathcal C})\rightarrow KK_{p+q} ({\mathcal B},{\mathcal C})$$

Given homomorphisms $\alpha \colon {\mathcal A}\rightarrow {\mathcal
B}$ and $\beta \colon {\mathcal B}\rightarrow {\mathcal C}$, the
product $[\alpha ]\cdot [\beta ]$ is the equivalence class of the
composition, $[\beta \circ \alpha ]$.

An essentially abstract argument, as described in section 6.34 of
\cite{CoT} for $R$-algebras rather than $R$-algebroids, yields the
following result.

\begin{theorem} \label{les}
Let
$$0\rightarrow {\mathcal A}\rightarrow {\mathcal B}\rightarrow {\mathcal C}\rightarrow 0$$
be an $F$-split short exact sequence of $R$-algebroids.  Let
$\mathcal D$ be an $R$-algebroid.  Then we have natural maps
$\partial \colon KK_p ({\mathcal A},{\mathcal D})\rightarrow
KK_{p+1}({\mathcal C},{\mathcal D})$ inducing a long exact sequence
of $KK$-theory groups
$$\rightarrow KK_p ({\mathcal C},{\mathcal D})\rightarrow KK_p ({\mathcal B},{\mathcal D})\rightarrow KK_p ({\mathcal A},{\mathcal D})\stackrel{\partial}{\rightarrow} KK_{p+1} ({\mathcal C},{\mathcal D})\rightarrow$$
\noproof
\end{theorem}

A similar result holds in the other variable; we do not need it
here.

\begin{definition}
We call a homomorphism $\alpha \colon J^{2n}{\mathcal A}\rightarrow
{\mathcal B}^{S^n}_\oplus$ a {\em $KK$-equivalence} if there is an
element
$[\alpha ]^{-1}\in KK_0 ({\mathcal B},{\mathcal A})$ such that
$[\alpha ]^{-1} \cdot [\alpha ] = [1_{\mathcal A}]$ and
$[\alpha] \cdot [\alpha ]^{-1} = [1_{\mathcal B}]$.

We call two $R$-algebroids $\mathcal A$ and $\mathcal B$ {\em
$KK$-equivalent} if there is a $KK$-equivalence $\alpha \colon
J^{2n}{\mathcal A}\rightarrow {\mathcal B}^{S^n}_\oplus$.
\end{definition}

\begin{proposition} \label{KKeq}
Let $\alpha \colon J^{2n} {\mathcal A}\rightarrow {\mathcal
B}^{S^n}_\oplus$ be a $KK$-equivalence.  Then the product with
$\alpha$ induces equivalences of spectra$$\KK ({\mathcal
B},{\mathcal C})\rightarrow \KK ({\mathcal A},{\mathcal C}) \qquad
\KK ({\mathcal C},{\mathcal A})\rightarrow \KK ({\mathcal
C},{\mathcal B})$$for every $R$-algebroid $\mathcal C$.
\end{proposition}

\begin{proof}
Let us just look at the map
$$\alpha \sharp \colon \KK ({\mathcal B},{\mathcal C})\rightarrow \KK ({\mathcal A},{\mathcal C})$$
since the other case is almost identical.  The map $\alpha \sharp$
induces an isomorphism
$$\pi_0 \KK ({\mathcal B},{\mathcal C}) = KK_0 ({\mathcal B},{\mathcal C})\rightarrow KK_0 ({\mathcal A},{\mathcal C})  =\pi_0 \KK ({\mathcal A},{\mathcal C})$$
by definition of the $KK$-theory group and the term
$KK$-equivalence.

Let $p\geq 0$.  Replacing the $R$-algebroid $\mathcal C$ by the
$R$-algebroid $\Omega^p {\mathcal C}$, we see that the map $\alpha
\sharp$ also induces an isomorphism
$$\pi_p \KK ({\mathcal A},{\mathcal C}) = KK_p ({\mathcal B},{\mathcal C})\rightarrow KK_p ({\mathcal A},{\mathcal C})  =\pi_p \KK ({\mathcal A},{\mathcal C})$$

Thus the map $\alpha \sharp$ is an equivalence of spectra as
desired, and we are done.
\end{proof}

Observe that the category-theoretic concept of natural isomorphism
makes sense when we are talking about homomorphisms of {\em unital}
$R$-algebroids.

\begin{lemma} \label{nisog}
Let $\alpha , \beta \colon {\mathcal A}\rightarrow {\mathcal B}$ be
naturally isomorphic homomorphisms of unital $R$-algebroids.  Then
the maps $\alpha$ and $\beta$ are simplicially homotopic at the
level of $KK$-theory spectra.
\end{lemma}

\begin{proof}
By definition of the $KK$-theory spectrum, the space $\KK ({\mathcal
A},{\mathcal B})_0$ is the simplicial set $\HOM ({\mathcal
A},{\mathcal B}_H )$.  In this space, the homomorphisms
$\alpha$ and $\beta$ are the same as the homomorphisms $\alpha'
\colon {\mathcal A}\rightarrow {\mathcal B}_\oplus$ and $\beta'
\colon {\mathcal A}\rightarrow {\mathcal B}_\oplus$ defined by
writing
\[
\alpha'(a) = \alpha (a) \oplus \beta (a),\ \beta'(a) = \alpha (a)
\oplus \beta (a) \qquad a\in \Ob ({\mathcal A})
\]
and
\[
\alpha' (x) = \left( \begin{array}{cc}
\alpha (x) & 0 \\
0 & 1 \\
\end{array} \right) \qquad
\beta' (x) = \left( \begin{array}{cc}
1 & 0 \\
0 & \beta (x) \\
\end{array} \right)
\]
where $x\in \Hom (a,b)_{\mathcal A}$.

Since the homomorphisms $\alpha$ and $\beta$ are naturally
equivalent, we can find invertible morphisms $g_a \in \Hom (\alpha
(a),\beta (a))_{\mathcal B}$ for each object $a\in \Ob ({\mathcal
A})$ such that $\beta (x)g_a = g_a^{-1} \alpha (x)$ for all $x\in
\Hom (a,b)_{\mathcal A}$.

Let
$$W = \left( \begin{array}{cc}
1-t^2 & (t^3 -2t)g_a^{-1} \\
tg_a & 1-t^2 \\
\end{array} \right)$$

Then the matrix $W$ is an invertible morphism in the $R$-algebroid
$({\mathcal A}\otimes {\mathbb Z}[t])_\oplus$; the inverse is the
matrix
$$W^{-1} = \left( \begin{array}{cc}
1-t^2 & (2t-t^3)g_a^{-1} \\
-tg_a & 1-t^2 \\
\end{array} \right)$$

Further,
$$e_0 (W) = \left( \begin{array}{cc}
1 & 0 \\
0 & 1 \\ \end{array} \right) = e_0 (W^{-1}) \qquad ev_1 (W) = \left(
\begin{array}{cc}
0 & -g_a^{-1} \\
g_a & 0 \\ \end{array} \right) = -e_1 (W^{-1})$$

Hence
$$e_0 (W \left( \begin{array}{cc}
\alpha (x) & 0 \\
0 & 1 \\
\end{array} \right) W^{-1}) = \alpha' (x)$$
and
$$e_1 (W \left( \begin{array}{cc}
\alpha (x) & 0 \\
0 & 1 \\
\end{array} \right) W^{-1}) = \beta' (x)$$

Therefore the homomorphisms $\alpha'$ and $\beta'$ are algebraically
homotopic.  The result now follows by theorem \ref{HOM}.
\end{proof}

The following result is immediate from the above lemma and
proposition \ref{KKeq}.

\begin{theorem} \label{equivalent2}
Let $\mathcal A$ and ${\mathcal A}'$ be equivalent unital
$R$-algebroids.  Let $\mathcal B$ be another $R$-algebroid.  Then
the spectra $\KK ({\mathcal A},{\mathcal B})$ and $\KK ({\mathcal
A}',{\mathcal B})$ are stably equivalent, and the spectra $\KK
({\mathcal B},{\mathcal A})$ and $\KK ({\mathcal B},{\mathcal A}')$
are homotopy-equivalent. \noproof
\end{theorem}

The following result is proved similarly.

\begin{proposition} \label{plus2}
The symmetric spectra $\KK ({\mathcal A},{\mathcal B})$ and $\KK
({\mathcal A},{\mathcal B}_\oplus )$ are stably equivalent. \noproof
\end{proposition}

Let $R$ be a ring where every $R$-algebra is central.  It is shouwn
in \cite{CoT} that for an $R$-algebra $A$ there is a natural
isomorphism $KK_n (R,{\mathcal A}) \cong KH_n ({\mathcal A})$.  The
proof depends on certain universal properties of $KK$-theory and
homotopy $K$-theory, and a universal characterisation of algebraic
$KK$-theory of the type first considered for $C^\ast$-algebras in
\cite{Hig,Th}.  The characterisation involves triangulated
categories; see \cite{MN,Nee}.

The following result is now immediate.

\begin{corollary}
Let $R$ be a ring where every $R$-algebra is central. Let $\mathcal
A$ be an $R$-algebroid that is equivalent to an $R$-algebra.  Then
we have a natural isomorphism
$$KK_n (R,{\mathcal A}) \cong KH_n ({\mathcal A})$$
\noproof
\end{corollary}

\begin{definition}
Let $\mathcal A$ be an $R$-algebroid.  Then we define the {\em
homotopy algebraic $K$-theory spectrum}
\[ \KH ({\mathcal A}) = \KK (R, {\mathcal A}) \]
\end{definition}

\section{Modules over Algebroids}

The modules we consider in this section were introduced for algebras in
\cite{Kass}.

\begin{definition}
Let $\mathcal A$ be an $R$-algebroid.  Then a {\em right $\mathcal
A$-module} is an $R$-linear contravariant functor from $\mathcal A$
to the category of $R$-modules.

A natural transformation, $T\colon {\mathcal E}\rightarrow {\mathcal
F}$, between two $\mathcal A$-modules is called a {\em
homomorphism}.
\end{definition}

We write ${\mathcal L}({\mathcal A})$ to denote the category of all
right ${\mathcal A}$-modules and homomorphisms.  The category
${\mathcal L}({\mathcal A})$ is clearly a unital
$R$-algebroid.\footnote{Assuming we take all of our right $\mathcal
A$-modules in a given universe, so that the category ${\mathcal
L}({\mathcal A})$ is small.}

\begin{definition}
Given an object $c\in \Ob ({\mathcal A})$, we define $\Hom
(-,c)_{\mathcal A}$ to be the right $\mathcal A$-module with spaces
$\Hom (-,c)_{\mathcal A} (a) = \Hom (a,c)_{\mathcal A}$; the action
of the $R$-algebroid $\mathcal A$ is defined by multiplication.
\end{definition}

\begin{definition}
Let $\mathcal E$ and $\mathcal F$ be right $\mathcal A$-modules.
Then we define the {\em direct sum}, ${\mathcal E}\oplus {\mathcal
F}$, to be the right $\mathcal A$-module with spaces
$$({\mathcal E}\oplus {\mathcal F}) (a)={\mathcal E}(a)\oplus {\mathcal F}(a) \qquad a\in \Ob ({\mathcal A})$$

We call a right $\mathcal A$-module $\mathcal E$ {\em finitely
generated} and {\em projective} if there is a right $\mathcal
A$-module $\mathcal F$ such that the direct sum ${\mathcal E}\oplus
{\mathcal F}$ is isomorphic to a direct sum of the form
$$\Hom (-,c_1)_{\mathcal A}\oplus \cdots \oplus \Hom (-,c_n)_{\mathcal A}$$
\end{definition}

Let us write ${\mathcal L}({\mathcal A}_\mathrm{fgp})$ to denote the
category of all finitely-generated projective $\mathcal A$-modules
and homomorphisms.   The following result follows directly from
theorem \ref{equivalent2} and proposition \ref{plus2}.

\begin{proposition} \label{fgpalg}
Let $\mathcal A$ and $\mathcal B$ be $R$-algebroids.  Then there is
a natural stable equivalence of spectra
$$\KK ({\mathcal A},{\mathcal B})\rightarrow \KK ({\mathcal A},{\mathcal L}({\mathcal B}_\mathrm{fgp}))$$
\noproof
\end{proposition}

Let $f\colon {\mathcal A}\rightarrow {\mathcal B}$ be a homomorphism
of $R$-algebroids, and let $\mathcal E$ be a right $\mathcal
A$-module.  Then we can define a right $\mathcal B$-module, $f_\ast
{\mathcal E} = {\mathcal B}\otimes_{\mathcal A}{\mathcal E}$ using a
similar method to that used for Hilbert modules over
$C^\ast$-algebras (see for example \cite{Lan}). Actually, the
procedure is slightly easier since we do not need to worry about the
analytic issues that are present in the $C^\ast$-algebra case.

\begin{definition}
Let $\mathcal A$ and $\mathcal B$ be $R$-algebroids.  Then an {\em
$({\mathcal A},{\mathcal B})$-bimodule}, $\mathcal F$, is an $R$-algebroid
homomorphism ${\mathcal F}\colon {\mathcal A} \rightarrow {\mathcal
L}({\mathcal B})$.  An $({\mathcal A},{\mathcal B})$-bimodule $\mathcal F$ is
termed {\em finitely generated and projective} if it is a
homomorphism
${\mathcal
F}\colon {\mathcal A} \rightarrow {\mathcal L}({\mathcal
B}_\mathrm{fgp})$.
\end{definition}

Let $\mathcal F$ be an $({\mathcal A},{\mathcal B})$-bimodule.  For
each object $a\in \Ob ({\mathcal A})$, let us write ${\mathcal
F}(-,a)$ to denote the associated right ${\mathcal B}$-module.  Then
for each object $b\in \Ob ({\mathcal B})$ we have an $R$-module
${\mathcal F}(b,a)$.  For each morphism $x\in \Hom (a,a')_{\mathcal
A}$, we have an induced homomorphism $x \colon {\mathcal
F}(-,a)\rightarrow {\mathcal F}(-,a')$.

Given an $R$-algebroid homomorphism $f\colon {\mathcal A}\rightarrow
{\mathcal B}$, there is an associated $({\mathcal A},{\mathcal
B})$-bimodule which we will simply label $\mathcal B$.  The right
$\mathcal B$-module associated to the object $a\in \Ob ({\mathcal
A})$ is $\Hom (-,f(a))_{\mathcal B}$.  The homomorphisms associated
to morphisms in the category $\mathcal A$ are defined in the obvious
way through the functor $f$.

\begin{definition}
Let $\mathcal E$ be a $\mathcal A$-module, and let $\mathcal F$ be a
$({\mathcal A},{\mathcal B})$-bimodule.  Let $b\in \Ob ({\mathcal
B})$.  Then we define the {\em algebraic tensor product} ${\mathcal
E}\otimes_{\mathcal A}{\mathcal F}(b,-)$ to be the $R$-module
consisting of all formal linear combinations
$$\lambda_1 (\eta_1 ,\xi_1) + \cdots +\lambda_n (\eta_n ,\xi_n )$$
where $\lambda_i \in R$, $\eta_i \in {\mathcal E}(a)$, $\xi_i \in
{\mathcal F}(b,a)$, and $a\in \Ob ({\mathcal A})$, modulo the
equivalence relation $\sim$ defined by writing

\begin{itemize}

\item $(\eta_1 + \eta_2 ,\xi ) \sim (\eta_1 , \xi ) + (\eta_2 ,\xi )$

\item $(\eta, \xi_1 + \xi_2 )\sim (\eta , \xi_1 ) + (\eta , \xi_2 )$

\item $(\eta , x\xi ) \sim (\eta x, \xi )$ for each morphism $x\in
\Hom (a,a')_{\mathcal A}$.

\end{itemize}

\end{definition}

Let us write $\eta \otimes \xi$ to denote the equivalence class of
the pair $(\eta , \xi )$.

\begin{definition}
We write ${\mathcal E}\otimes_{\mathcal A}{\mathcal F}$ to denote
the right $\mathcal B$-module defined by associating the $R$-module
${\mathcal E}\otimes_{\mathcal A}{\mathcal F}(b,-)$ to an object
$b\in \Ob ({\mathcal B})$.
\end{definition}

Given a homomorphism of right $\mathcal A$-modules, $T\colon
{\mathcal E}\rightarrow {\mathcal E}'$, there is an induced
homomorphism $T\otimes 1 \colon {\mathcal E}\otimes_{\mathcal
A}{\mathcal F}\rightarrow {\mathcal E}'\otimes_{\mathcal A}{\mathcal
F}$, defined in the obvious way.

Recall that a homomorphism $f\colon {\mathcal A}\rightarrow
{\mathcal B}$ turns the $R$-algebroid $\mathcal B$ into an
$({\mathcal A},{\mathcal B})$-bimodule.  Thus, given an $\mathcal
A$-module $\mathcal E$, we can form the tensor product ${\mathcal
E}\otimes_{\mathcal A}{\mathcal B}$.

We can write $f_\ast {\mathcal E} = {\mathcal E}\otimes_{\mathcal
A}{\mathcal B}$.  A homomorphism of $\mathcal A$-modules, $T\colon
{\mathcal E}\rightarrow {\mathcal E}'$ yields a functorially induced
map $f_\ast T = T\otimes 1\colon {\mathcal E}\otimes_{\mathcal
A}{\mathcal B}\rightarrow {\mathcal E}'\otimes_{\mathcal A}{\mathcal
B}$.

\begin{theorem} \label{Mproducta}
Let $\mathcal E$ be a finitely generated projective right $\mathcal
A$-module.  Then we have a canonical morphism of spectra
$${\mathcal E}\wedge \colon \KK ({\mathcal A}, {\mathcal B})\rightarrow \KK (R,{\mathcal B})$$

This morphism is natural in the variable ${\mathcal B}$ in the
obvious sense.  Given an $R$-algebroid homomorphism $f\colon
{\mathcal A}\rightarrow {\mathcal A}'$, we have a commutative
diagram
$$\begin{array}{cccc}
{\mathcal E}\wedge \colon & \KK ({\mathcal A},{\mathcal B}) & \rightarrow & \KK (R,{\mathcal B}) \\
& \uparrow & & \| \\
{\mathcal E}\otimes_{\mathcal A}{\mathcal A}'\wedge \colon & \KK ({\mathcal A}',{\mathcal B}) & \rightarrow & \KK (R,{\mathcal B}) \\
\end{array}$$
\end{theorem}

\begin{proof}
The right $\mathcal A$-module $\mathcal E$ defines a homomorphism
$R\rightarrow {\mathcal L}({\mathcal A}_\mathrm{fgp})$ by mapping
the one object of $R$, considered as an $R$-algebroid, to the
Hilbert module $\mathcal E$, and the unit $1$ to the identity
homomorphism on $\mathcal E$.  Thus the Hilbert module $\mathcal E$
defines an element in the $0$-th space of the spectrum $\KK
(R,{\mathcal L}({\mathcal A}_\mathrm{fgp}))$.

By proposition \ref{fgpalg}, we have a natural equivalence of
symmetric spectra
$$\KK (R,{\mathcal L}({\mathcal A}_\mathrm{fgp}))\rightarrow \KK (R,{\mathcal A})$$

By composition of the formal inverse of the above equivalence with the product, we obtain a canonical map
$${\mathcal E}\wedge \colon \KK ({\mathcal A}, {\mathcal B})\rightarrow \KK (R,{\mathcal B})$$

Naturality in the variable $\mathcal B$ follows by associativity of
the product.  To prove that the stated diagram is commutative, we
need to prove that the diagram
$$\begin{array}{cccc}
{\mathcal E}\wedge \colon & \KK ({\mathcal L}({\mathcal A}_\mathrm{fgp}),{\mathcal B}) & \rightarrow & \KK (R,{\mathcal B}) \\
& \uparrow & & \| \\
{\mathcal E}\otimes_{\mathcal A}{\mathcal A}'\wedge \colon & \KK ({\mathcal L}({\mathcal A}'_\mathrm{fgp}),{\mathcal B}) & \rightarrow & \KK (R,{\mathcal B}) \\
\end{array}$$
is commutative.

The homomorphism $f_\ast \colon {\mathcal L}({\mathcal
A}_\mathrm{fgp})\rightarrow {\mathcal L}({\mathcal
A}'_\mathrm{fgp})$ is defined by mapping the right $\mathcal
A$-module $\mathcal E$ to the right $\mathcal A$-module ${\mathcal
E}\otimes_{\mathcal A}{\mathcal A}'$, and the homomorphism $T\colon
{\mathcal E}\rightarrow {\mathcal F}$ to the homomorphism $T\otimes
1 \colon {\mathcal E}\otimes_{\mathcal A}{\mathcal A}' \rightarrow
{\mathcal F}\otimes_{\mathcal A}{\mathcal A}'$.

Now, the map at the bottom of the diagram is defined by the product
with the homomorphism $R\rightarrow {\mathcal L}({\mathcal
A}'_\mathrm{fgp})$ defined by the right ${\mathcal A}'$-module
${\mathcal E}\otimes_{\mathcal A}{\mathcal A}'$.  The composition of
the vertical map on the left and the map at the top of the diagram
is defined by the product with the composition of the homomorphism
$f_\ast$ and the homomorphism ${\mathbb C}\rightarrow {\mathcal
L}({\mathcal A})$ defined by the right ${\mathcal A}'$-module
${\mathcal E}$.

By construction, these two homomorphisms are the same, and we are
done.
\end{proof}

We are actually going to use a slight generalisation of the above
theorem; the proof is essentially the same as the above.

\begin{definition}
Let $\mathcal A$ be an $R$-algebroid.  Let $a\in \Ob ({\mathcal
A})$.  Then we define the {\em path-component} of $a$, ${\mathcal
A}|_{\Or (a)}$, to be the full subcategory of $\mathcal A$
containing all objects $b\in \Ob ({\mathcal A})$ such that $\Hom
(a,b)_{\mathcal A} \neq 0$.

We call a right $\mathcal A$-module $\mathcal E$ {\em almost
finitely generated and projective} if the restriction to each
path-component is finitely generated and projective.
\end{definition}

\begin{theorem} \label{Mproduct2a}
Let $\mathcal E$ be an almost finitely generated projective
$\mathcal A$-module.  Then we have a canonical morphism of spectra
$${\mathcal E}\wedge \colon \KK ({\mathcal A}, {\mathcal B})\rightarrow \KK (R,{\mathcal B})$$

This morphism is natural in the variable ${\mathcal B}$ in the
obvious sense.  Given an $R$-algebroid homomorphism $f\colon
{\mathcal A}\rightarrow {\mathcal A}'$, we have a commutative
diagram
$$\begin{array}{cccc}
{\mathcal E}\wedge \colon & \KK ({\mathcal A},{\mathcal B}) & \rightarrow & \KK (R,{\mathcal B}) \\
& \uparrow & & \| \\
{\mathcal E}\otimes_{\mathcal A}{\mathcal A}'\wedge \colon & \KK ({\mathcal A}',{\mathcal B}) & \rightarrow & \KK (R,{\mathcal B}) \\
\end{array}$$
\noproof
\end{theorem}

\section{Equivariant $KK$-theory}

Let $\mathcal G$ be a discrete groupoid, and let $R$ be a ring.  We
can regard $\mathcal G$ as a small category in which every morphism
is invertible.

\begin{definition}
A {\em
$\mathcal G$-algebra} over $R$ is a functor from the category
$\mathcal G$ to the category of $R$-algebras and homomorphisms.
\end{definition}

Thus, if $A$ is a $\mathcal G$-algebra, then for each object $a\in
\Ob ({\mathcal G})$ we have an algebra $A(a)$.  A morphism $g\in
\Hom (a,b)_{\mathcal G}$ induces a homomorphism $g\colon
A(a)\rightarrow A(b)$.

We can regard an ordinary algebra $C$ as a $\mathcal
G$-algebra by writing $C(a)=C$ for each object $a\in \Ob
({\mathcal G})$ and saying that each morphism in the groupoid
$\mathcal G$ acts as the identity map.

\begin{definition}
A {\em $\mathcal G$-module over $R$} is a
functor from the groupoid $\mathcal G$ to the category of
$R$-bimodules and $R$-linear maps.
\end{definition}

The notation used for $\mathcal G$-modules is the same as that used for $\mathcal G$-algebras. There
is a forgetful functor, $F$, from the category of $\mathcal
G$-algebras to the category of $\mathcal G$-modules.

An {\em equivariant map} between $\mathcal G$-algebras or $\mathcal
G$-modules is the same thing as a natural transformation.

\begin{definition}
Let $A$ and $B$ be $\mathcal G$-algebras.  Then we define the {\em
tensor product} $A\otimes_R B$ to be the $\mathcal G$-algebra where
$(A\otimes_R B)(a) = A(a)\otimes_R B(a)$ for each object $a\in \Ob ({\mathcal G})$, and the $\mathcal
G$-action is defined by writing $g(x\otimes y) = g(x)\otimes g(y)$
whenever $g\in \Hom (a,b)_{\mathcal G}$, $x\in A(a)$, and $y\in
B(a)$.

We define the {\em direct sum} $A\oplus B$ to be the $\mathcal
G$-algebra where $(A\oplus B)(a)$ is the direct sum $A(a)\oplus
B(a)$ for each object $a\in \Ob ({\mathcal G})$ and the $\mathcal
G$-action is defined by writing $g(x\oplus y)= g(x)\oplus g(y)$
whenever $g\in \Hom (a,b)_{\mathcal G}$, $x\in A(a)$, and $y\in
B(a)$.
\end{definition}

We similarly define tensor products and direct sums of $\mathcal
G$-modules.

Recall from section \ref{alghom} that we have a sequence of rings,
$(\Z^{\Delta^n})$; these rings combine to form a simplicial ring
$\Z^\Delta$.  If we equip each such ring with the trivial
$\mathcal G$-action, and $A$ is a $\mathcal G$-algebra, we can form the
tensor product
$A\otimes_{\Z}{\Delta^n}$.

Just as in section \ref{alghom}, given a $\mathcal G$-algebra $A$
and a simplicial set $X$, we can form the $\mathcal G$-algebra.
$A^X$.  If the simplicial set $X$ has a basepoint $+$, we also form the $\mathcal G$-algebra $A^{(X,+)} = \ker (A^X \rightarrow A)$.

\begin{definition}
Let $f_0 , f_1 \colon {A}\rightarrow {B}$ be equivariant maps of
$\mathcal G$-algebras.  An equivariant map $h\colon {\mathcal
A}\rightarrow {\mathcal B}^{\Delta^1}$ such that $e_0 \circ h =f_0$
and $e_1 \circ h = f_1$ is called an {\em elementary homotopy}
between the maps $f_0$ and $f_1$.

We call $f_0$ and $f_1$ {\em algebraically homotopic} if they can be
linked by a chain of elementary homotopies.  We write
$[A,B]_{\mathcal G}$ to denote the set of algebraic homotopy
classes.
\end{definition}

Let us call a directed object in the category of $\mathcal
G$-algebras and equivariant maps a {\em directed $\mathcal
G$-algebra}.  As in the non-equivariant case, simplicial subdivision gives us a
directed $\mathcal G$-algebra $A^{\sd^\bullet X}$.

\begin{definition}
Let $A$ and $B$ be $\mathcal G$-algebras.  Then we define
$\HOM_{\mathcal G} ({A},{B})$ to be the simplicial set defined by
writing
$$\HOM_{\mathcal G} ({A},{B})[n] = \lim_{\rightarrow \atop k} \Hom ({A}, {B}^{\sd^k \Delta^n})$$

The face and degeneracy maps are those inherited from the simplicial
$\mathcal G$-algebra ${B}^\Delta$.
\end{definition}

The following result is proved in the same way as theorem 3.3.2 in
\cite{CoT}.

\begin{theorem}
Let $A$ be a $\mathcal G$-algebra, and let $B$ be a directed
$\mathcal G$-algebra.  Then
$$[{A},{B}^{S^n}]_{\mathcal G} = \pi_n \HOM_{\mathcal G} ({A},{B})$$
\noproof
\end{theorem}

A {\em short exact sequence} of $\mathcal
G$-algebras is a sequence of $\mathcal G$-algebras and
equivariant maps
\[
0\rightarrow A\stackrel{i}{\rightarrow} B \stackrel{j}{\rightarrow}
C\rightarrow 0
\]
such that the sequence
\[
0\rightarrow A(a)\stackrel{i}{\rightarrow}
B(a)\stackrel{j}{\rightarrow} C(a)\rightarrow 0
\]
is exact for each object $a\in \Ob ({\mathcal G})$.  A splitting of
a short exact sequence is defined in the obvious way.

We call a short exact sequence
\[
0\rightarrow A\rightarrow B\stackrel{j}{\rightarrow C}\rightarrow 0
\]
{\em $F$-split} if there is an equivariant map {\em of $\mathcal
G$-modules} $s\colon C\rightarrow B$ such that $j\circ s =1_C$.  Such
a map $s$ is called an {\em $F$-splitting}.

\begin{definition}
Let $A$ be a $\mathcal G$-module.  Then we define the {\em tensor
$\mathcal G$-algebra} $A^{\otimes k}$ to be the tensor product of
$A$ with itself $k$ times.  We define the {\em equivariant tensor
algebra}, $TA$, to be the iterated direct sum
\[
TA = \oplus_{k=1}^\infty A^{\otimes k}
\]
\end{definition}

Composition of elements in the tensor $\mathcal G$-algebra is
defined by concatenation of tensors.  Formation of the tensor
$\mathcal G$-algebra defines a functor, $T$ from the category of
$\mathcal G$-modules to the category of $\mathcal G$-algebras.  The
following result is proved in the same way as proposition
\ref{TFadjoint}.

\begin{proposition} \label{TFadjointe}
The functor $T$ is naturally adjoint to the forgetful functor $F$.
\noproof
\end{proposition}

Given a $\mathcal G$-algebra $\mathcal A$, there is a canonical
equivariant map $\sigma \colon A\rightarrow TA$ {\em of
$\mathcal G$-modules} defined by mapping each morphism set of the $\mathcal
G$-algebra $A$ onto the first summand.  As we should by now expect,
there is an associated universal property.

\begin{proposition}
Let $\alpha \colon A\rightarrow B$ be an equivariant map between
$\mathcal G$-algebras.  Then there is a unique homomorphism $\varphi
\colon TA\rightarrow B$ such that $\alpha = \varphi \circ \sigma$.
\noproof
\end{proposition}

We can thus define a $\mathcal G$-algebra $JA$ by writing
\[
JA(a) = \ker \pi \colon TA(a)\rightarrow A(a)
\]
for each object $a\in \Ob ({\mathcal G})$.  The $\mathcal G$-action
is inherited from the tensor $\mathcal G$-algebra.  There is a
natural short exact sequence
\[
0\rightarrow J{A} \hookrightarrow T{A} \stackrel{\pi}{\rightarrow}
{A}\rightarrow 0
\]
with $F$-splitting $\sigma \colon {A}\rightarrow T{A}$.  The
following result is proved in the same way as theorem \ref{ufsplit}.

\begin{theorem} \label{ufsplite}
Let
$$0\rightarrow { I}\stackrel{i}{\rightarrow} { E} \stackrel{j}{\rightarrow} { A}\rightarrow 0$$
be an $F$-split short exact sequence.  Then we have a natural
homomorphism $\gamma \colon J{ A}\rightarrow { I}$ fitting into a
commutative diagram
$$\begin{array}{ccccccccc}
0 & \rightarrow & J{ A} & \rightarrow & T{ A} & \rightarrow & { A} & \rightarrow & 0 \\
& & \downarrow & & \downarrow & & \| & & \\
0 & \rightarrow & { I} & \rightarrow & { E} & \rightarrow & { A} & \rightarrow & 0 \\
\end{array}$$
\noproof
\end{theorem}

As before, the homomorphism $\gamma$ is called the {\em classifying
map} of the short exact sequence.

The last construction we need to define an equivariant $KK$-theory
spectrum is an equivariant version of the path extension.
Analogously to the algebroid case, we can write $\Omega {A} =
{A}^{S^1 ,+}$ and obtain an $F$-split short exact sequence
$$0\rightarrow \Omega {A} \rightarrow {A}^{\Delta^1} \stackrel{(e_0,e_1)}{\rightarrow} {A}\oplus {A}\rightarrow 0$$

As before, a diagram chase enables us to construct a natural classifying map $\rho \colon JA \rightarrow \Omega
A$.

\begin{definition}
Let $f\colon {A}\rightarrow {B}$ be an equivariant map of $\mathcal
G$-algebras.  Then the {\em path algebra} of $f$ is the unique
$\mathcal G$-algebra $P{B}\oplus_{B}{A}$ fitting into a commutative
diagram
$$\begin{array}{ccccc}
\Omega {B} & \rightarrow & P{B}\oplus_{B}{A} & \rightarrow & {A} \\
\| & & \downarrow & & \downarrow \\
\Omega {B} & \rightarrow & P{B} & \rightarrow & {B} \\
\end{array}$$
where the bottom row is the path extension, and the upper row is a
short exact sequence.
\end{definition}

It is straightforward to check that the path algebra is
well-defined, and that the upper row in the above diagram has a
natural $F$-splitting.  There is therefore a naturalmap
$\eta (f)\colon J{A}\rightarrow \Omega {B}$.

We can compose with the map $i\colon \Omega {A}\rightarrow
{A}^{S^1}$ to obtain a map $\eta (f) \colon J{A} \rightarrow
{B}^{S^1}$.

\begin{definition}
We call the above equivariant map $\eta (f)$ the {\em classifying
map} of the equivariant map $f$.
\end{definition}

Given a $\mathcal G$-algebra $A$, let us write $M_\infty (A)$ to
denote the direct limit of the $\mathcal G$-algebras $M_n (A)=
A\otimes_R M_n (R)$ under the inclusions
$$x\mapsto \left( \begin{array}{cc}
x & 0 \\
0 & 0 \\ \end{array} \right)$$

Consider an equivariant map $\alpha \colon A\rightarrow M_\infty
(B^{S^n})$, where $A$ and $B$ are $\mathcal G$-algebras.  Then by
the above construction, there is an induced classifying map $\eta
(\alpha )\colon J{A} \rightarrow M_\infty ({B}^{S^{n+1}})$.

\begin{definition}
We define $\KK_{\mathcal G}({A},{B})$ to be the symmetric spectrum
with sequence of spaces $\HOM_{\mathcal G} (J^{2n} {A},M_\infty
({B}^{S^n}))$.  The structure map
\[
\epsilon \colon \HOM_{\mathcal G} (J^{2n} {A},M_\infty
({B}^{S^n}))\rightarrow \Omega \HOM_{\mathcal G} (J^{2n+2}
{A},M_\infty ({B}^{S^{n+1}})) \cong \HOM_{\mathcal G} (J^{2n+2}
{A},M_\infty ({B}^{S^{n+2}}))
\]
is defined by applying the above classifying map construction twice,
that is to say writing $\epsilon (\alpha ) = \eta (\eta (\alpha ))$
whenever $\alpha \in \HOM (J^{2n} {A}, M_\infty ({B}^{S^{n+1}}))$.

The action of the permutation group $S_n$ is induced by its
canonical action on the simplicial sphere $S^n$.
\end{definition}

The product is constructed as for $R$-algebroids, in the
non-equivariant case.

\begin{definition} \label{Shprodalge}
Let $A$, $B$, and $C$ be $\mathcal G$-algebras.  Let $\alpha \in
\HOM_{\mathcal G} (J^{2m} {A}, M_\infty ({B}^{S^m}))$ and $\beta \in
\HOM_{\mathcal G} (J^{2n} {B},M_\infty ({C}^{S^n}))$.  Then we
define the product $\alpha \sharp \beta$ to be the composition
\[
J^{2m+2n}{A}\stackrel{J^{2n} \alpha}{\rightarrow} J^{2n}(M_\infty
({B}^{S^m}) \rightarrow (J^{2n}M_\infty (B))^{S^m}
\stackrel{\beta}{\rightarrow} M_\infty ({C})^{S^{m+n}}
\]
\end{definition}

\begin{theorem} \label{KKprodalge}
Let $A$, $B$, and $C$ be $\mathcal G$-algebras.  Then there is a
natural map of spectra
\[
\KK_{\mathcal G}({A},{B}) \wedge \KK_{\mathcal G}({B},{C})
\rightarrow \KK_{\mathcal G}({A},{C})
\]
defined by the formula
\[
\alpha \wedge \beta \mapsto \alpha \sharp \beta \qquad \alpha \in
HOM_{\mathcal G} (J^{2m} {A}, M_\infty ({B}^{S^m}) ),\ \beta \in
\HOM_{\mathcal G} (J^{2n} {B}, M_\infty ({C}^{S^n}) )
\]

Further, the above product is associative in the usual sense.  Given
equivariant maps $\alpha \colon {A}\rightarrow {B}$ and $\beta
\colon {B}\rightarrow {C}$, we have the formula $\alpha \sharp \beta
= \beta \circ \alpha$. \noproof
\end{theorem}

Just as in the non-equivariant case, the following result follows
from our constructions.

\begin{corollary}
Let $A$ be a $\mathcal G$-algebra.  Then the spectrum $\KK_{\mathcal
G} ({A},{A})$ is a symmetric ring spectrum.

Let $B$ be another $\mathcal G$-algebra.  Then the spectrum
$\KK_{\mathcal G} (A,B)$ is a symmetric $\KK_{\mathcal G} (R ,R
)$-module spectrum. \noproof
\end{corollary}

Let $\theta \colon {\mathcal G}\rightarrow {\mathcal H}$ be a
functor between groupoids, and let $A$ be an $\mathcal H$-algebra.
Abusing notation, we can also regard $A$ as a $\mathcal G$-algebra;
we write $A(a)=A(\theta (a))$ for each object $a\in \Ob( {\mathcal
G})$, and define a homomorphism $g=\theta (g) \colon A(\theta
(a))\rightarrow A(\theta (b))$ for each morphism $g\in \Hom
(a,b)_{\mathcal G}$.

There is an induced map $\theta^\ast \colon \KK_{\mathcal
H}(A,B)\rightarrow \KK_{\mathcal G}(A,B)$ defined by the observation
that any $\mathcal H$-equivariant map is also $\mathcal
G$-equivariant.

\begin{definition}
We call the map $\theta^\ast \colon \KK_{\mathcal
H}(A,B)\rightarrow \KK_{\mathcal G}(A,B)$ {\em the restriction map.}
\end{definition}

\begin{proposition} \label{restriction2}
Let $\theta \colon {\mathcal G}\rightarrow {\mathcal H}$ be a
functor between groupoids, and let $A$ and $B$ be $\mathcal
H$-algebras. Then the restriction map $\theta^\ast
\colon \KK_{\mathcal H}(A,B)\rightarrow \KK_{\mathcal G}(A,B)$ is compatible with the product in the sense that we
have a commutative diagram
\[
\begin{array}{ccc}
\KK_{\mathcal H}(A,B)\wedge \KK_{\mathcal H}(B,C) & \rightarrow & \KK_{\mathcal H}(A,C) \\
\downarrow & & \downarrow \\
\KK_{\mathcal G}(A,B)\wedge \KK_{\mathcal G}(B,C) & \rightarrow & \KK_{\mathcal G}(A,C) \\
\end{array}
\]
where the horizontal map is defined by the product and the vertical
maps are restriction maps.
\end{proposition}

\begin{proof}
Abusing notion, we can also regard the $\mathcal H$-algebra $A$ as a
$\mathcal G$-algebra; we write $A(a)=A(\theta (a))$ for each object
$a\in \Ob( {\mathcal G})$, and define a homomorphism $g=\theta (g)
\colon A(\theta (a))\rightarrow A(\theta (b))$ for each morphism
$g\in \Hom (a,b)_{\mathcal G}$.

We similarly regard the $\mathcal H$-algebra $B$ as a $\mathcal
G$-algebra.  An $\mathcal H$-equivariant map $\alpha \colon
J^n{\mathcal A}\rightarrow M_\infty ({\mathcal B}^{S^n})$ is also
$\mathcal G$-equivariant.  We use this construction to define our
map $\theta^\ast \colon \KK_{\mathcal H}(A,B)\rightarrow
\KK_{\mathcal G}(A,B)$.

The result is now straightforward to check.
\end{proof}

Given a $\mathcal G$-algebra $A$, we define the {\em convolution
algebroid}, $A{\mathcal G}$, to be the algebroid with the same set
of objects as the groupoid $\mathcal G$, and morphism sets
\[
\Hom (a,b)_{A{\mathcal G}} = \{ \sum_{i=1}^m x_i g_i \ |\ x_i \in
A(b), g_i \in \Hom (a,b)_{\mathcal G},\ m\in \N \}
\]

Composition of morphisms is defined by the formula
\[
\left( \sum_{i=1}^m x_i g_i \right) \left( \sum_{j=1}^n y_j h_i
\right) = \sum_{i,j=1}^{m,n} x_i g_i(y_j) g_i h_j
\]

\begin{theorem} \label{descentealg}
Let $\mathcal G$ be a groupoid, and let $A$ and $B$ be $\mathcal
G$-algebras.  Then there is a map
\[
D\colon \KK_{\mathcal G} (A,B)\rightarrow \KK(A{\mathcal G},
B{\mathcal G})
\]
which is compatible with the product in the sense that we have a
commutative diagram
\[
\begin{array}{ccc}
\KK_{\mathcal G}({A},{B})\wedge \KK_{\mathcal G}({B},{C}) & \rightarrow & \KK (A,{C}) \\
\downarrow & & \downarrow \\
\KK ({A}{\mathcal G},{B}{\mathcal G})\wedge \KK ({B}{\mathcal G},{C}{\mathcal G}) & \rightarrow & \KK ({A}{\mathcal G},{C}{\mathcal G}) \\
\end{array}
\]
where the horizontal maps are defined by the product.
\end{theorem}

\begin{proof}
Let $\alpha \colon J^{2n}A\rightarrow M_\infty (B^{S^n})$ be an
equivariant map.  Then we have a functorially induced homomorphism
$\alpha_\ast \colon (J^{2n}A){\mathcal G} \rightarrow M_\infty
(B^{S^n}) {\mathcal G}$ defined by writing
$$\alpha_\ast (\sum_{i=1}^n x_i g_i ) = \sum_{i=1}^n \alpha (x_i) g_i$$

We have a natural homomorphism $\gamma \colon J(A{\mathcal
G})\rightarrow (JA) {\mathcal G}$ defined as the classifying map of
the diagram
\[
\begin{array}{ccccccccc}
& & & & & & {A}{\mathcal G} \\
& & & & & & \| \\
0 & \rightarrow & (J{A}){\mathcal G} & \rightarrow & (T{A}){\mathcal G} & \rightarrow & {A}{\mathcal G} & \rightarrow & 0 \\
\end{array}
\]

For any $\mathcal G$-algebra $C$, we can regard the $\mathcal
G$-algebra $M_\infty ({C})$ as the tensor product ${C}\otimes_R
M_\infty (R)$.  It follows that there is an obvious homomorphism
$\beta \colon M_\infty (B^{S^n}){\mathcal G} \rightarrow M_\infty
(B{\mathcal G})^{S^n}$.

We thus have a map $D\colon \KK_{\mathcal G} (A,B)\rightarrow \KK
(A {\mathcal G}, B{\mathcal G})$ defined by writing $D (\alpha ) = \beta \circ \alpha_\ast \circ \gamma^{2n}$. The relevant
naturality properties are easy to check.
\end{proof}

\begin{corollary}
Let $\mathcal G$ be a discrete groupoid, and let $A$ and $B$ be
$\mathcal G$-algebras.  Then the spectrum $\KK (A{\mathcal G},
B{\mathcal G})$ is a symmetric $\KK_{\mathcal G}(R ,R )$-module
spectrum. \noproof
\end{corollary}

Let $\mathcal G$ be a discrete groupoid, and let $A$ and $B$ be
$\mathcal G$-algebras.  Then we can define groups
$$KK_p^{\mathcal G} ({A},{B}) := \pi_p \KK_{\mathcal G} ({A},{B})$$

The following result is proved in the same way as theorem \ref{les}.

\begin{theorem} \label{les2}
Let
$$0\rightarrow {A}\rightarrow {B}\rightarrow {C}\rightarrow 0$$
be an $F$-split short exact sequence of $\mathcal G$-algebras.  Let
$D$ be an $R$-algebroid.  Then we have natural maps $\partial \colon
KK_p^{\mathcal G} ({A},{D})\rightarrow KK_{p+1}^{\mathcal
G}({C},{D})$ inducing a long exact sequence of $KK$-theory groups
$$\rightarrow KK_p^{{\mathcal G}} ({C},{D})\rightarrow KK_p^{{\mathcal G}} ({B},{D})
\rightarrow KK_p^{{\mathcal G}} ({A},{D})\stackrel{\partial}{\rightarrow}
KK_{p+1}^{{\mathcal G}} ({C},{D})\rightarrow$$
\noproof
\end{theorem}

\begin{theorem} \label{reseq}
Let $\theta \colon {\mathcal G}\rightarrow {\mathcal H}$ be an
equivalence of discrete groupoids.  Let $A$ and $B$ be $\mathcal
H$-algebras.  Then the restriction map $\theta^\ast
\colon \KK_{\mathcal H}(A,B)\rightarrow \KK_{\mathcal G}(A,B)$ is an
isomorphism of spectra.
\end{theorem}

\begin{proof}
Since the functor $\theta$ is an equivalence, there is a functor
$\phi \colon {\mathcal H}\rightarrow {\mathcal G}$ along with
natural isomorphisms $G\colon \phi \circ \theta \rightarrow
1_{\mathcal
G}$ and $H\colon \theta \circ \phi \rightarrow 1_{\mathcal H}$.

Thus, for each object $a\in \Ob ({\mathcal G})$, there is an
isomorphism $H_a \in \Hom (\phi \theta (a),a)_{\mathcal G}$.  Let
$\alpha \colon J^{2n}A\rightarrow M_ \infty (B^{S^n})$ be an $\mathcal
H$-equivariant map. Then the map $\alpha$ can be defined in terms of
the restriction $\phi^\ast \theta^\ast \alpha \colon
J^{2n}A\rightarrow
M_\infty (B^{S^n})$ by the formula
\[
\alpha (x) = \phi^\ast \theta^\ast \alpha (H_a^{-1} xH_a) \qquad
x\in A(a)
\]

Thus the equivariant map $\alpha$ is determined by the restriction
$\phi^\ast \theta^\ast \alpha$.  The natural isomorphism $H$
therefore induces a isomorphism of spectra $H_\ast \colon
\KK_{\mathcal H}(A,B)\rightarrow \KK_{\mathcal H}(A,B)$ such that
$H_\ast \circ \phi^\ast \circ \theta^\ast = 1_{\KK_{\mathcal
H}(A,B)}$.  There is similarly a isomorphism $G_\ast \colon
\KK_{\mathcal G}(A,B)\rightarrow \KK_{\mathcal G}(A,B)$ such that
$G_\ast \circ \theta^\ast \circ \phi^\ast = 1_{\KK_{\mathcal
G}(A,B)}$.

It follows that the map $\theta^\ast$ is a isomorphism, and we are
done.
\end{proof}

\section{Assembly}

Given a functor, $\mathbb E$, from the category of
$G$-$CW$-complexes to the category of spectra, we call
$\mathbb E$ {\em $G$-homotopy-invariant} if it takes
$G$-homotopy-equivalent equivariant maps of $G$-spaces to maps of
spectra that induce the same maps between stable homotopy groups.

We call the functor $\mathbb E$ {\em $G$-excisive} if it is
$G$-homotopy-invariant, and the collection of functors $X\mapsto
\pi_\ast {\mathbb E} (X)$ forms a $G$-equivariant homology
theory.\footnote{See for example chapter 20 of \cite{KL} for the
relevant definitions.}

\begin{definition} \label{GCW}
Let $G$ be a discrete group.  We define the {\em classifying space
for proper actions}, $\EG$, to be the $G$-$CW$-complex with the
following properies:

\begin{itemize}

\item For each point $x\in X$ the isotropy group
$$G_x = \{ g\in G \ |\ xg =x \}$$
is finite.\footnote{As a convention, if we mention a $G$-space, we
assume that the group $G$ acts on the right.}

\item For a
given subgroup $H\leq G$ the fixed point set $\EG^H$ is
equivariantly contractible if $H$ is finite, and empty otherwise.

\end{itemize}

\end{definition}

The classifying space $\EG$ always exists, and is unique up to
$G$-homotopy-equivalence; see \cite{BCH,DL}.

The following two results from \cite{DL} are the main abstract
results on assembly maps we need in this article.

\begin{theorem} \label{DL}
Let $G$ be a discrete group, and let $\mathbb E$ be a $G$-homotopy invariant functor from
the category of $G$-$CW$-complexes to the category of spectra.
Then there is a $G$-excisive functor ${\mathbb E}'$ and a natural
transformation $\alpha \colon {\mathbb E}'\rightarrow {\mathbb E}$
such that the map
$$\alpha \colon {\mathbb E}' (G/H)\rightarrow {\mathbb E}(G/H)$$
is a stable equivalence whenever $H$ is a finite subgroup of $G$.

Further, the pair $({\mathcal E}' , \alpha )$ is unique up to stable
equivalence. \noproof
\end{theorem}

\begin{definition}
Let $G$ be a discrete group.  Then we define the {\em orbit
category}, $\Or (G)$, to be the category in which the objects are
the $G$-spaces $G/H$, where $H$ is a subgroup of $G$, and the
morphisms are $G$-equivariant maps.

An {\em $\Or (G)$-spectrum} is a functor from the category $\Or (G)$
to the category of symmetric spectra.
\end{definition}

\begin{theorem} \label{DL2}
Let $\mathbb E$ be an $\Or (G)$-spectrum.  Then there is a
$G$-excisive functor, ${\mathbb E}'$, from the category of
$G$-CW-complexes to the category of spectra such that ${\mathbb E}'
(G/H) = {\mathbb E}(G/H)$ whenever $H$ is a subgroup of $G$.

Further, given a functor ${\mathbb F}$ from the category of
$G$-CW-complexes to the category of spectra, there is a natural
transformation
$$\beta \colon ({\mathbb F}|_{\Or (G)})' \rightarrow {\mathbb F}$$
such that the map
$$\beta \colon ({\mathbb F}|_{\Or (G)})'(G/H) \rightarrow {\mathbb F}(G/H)$$
is a stable equivalence whenever $H$ is a subgroup of the group $G$.
\noproof
\end{theorem}

The constant map $c\colon \EG \rightarrow +$ induces a
map $c_\ast \colon {\mathbb E}'(\EG )\rightarrow {\mathbb
E}(+)$.  This map is called the {\em assembly map}.  The corresponding
{\em isomorphism conjecture} is the assertion that this assembly map
is a stable equivalence.

\begin{definition}
Consider a group $G$, and a $G$-space $X$.  Then we define the {\em
transport groupoid}, $\overline{X}$, to be the groupoid in which the
set of objects is the space $X$, considered as a discrete set, and
we have morphism sets
$$\Hom (x,y)_{\overline{X}} = \{ g\in G \ | xg=y \}$$
\end{definition}

Composition of morphisms in the transport groupoid is defined by the
group operation.  There is a faithful functor $i\colon
\overline{X}\rightarrow G$ defined by the inclusion of each morphism
set in the group.

Given a ring $R$, a group $G$, and a $G$-algebra $A$ over $R$, let
$X$ be a $G$-$CW$-complex.  Then (through the functor $i$), the
$G$-algebra $A$ can also be considered an $\overline{X}$-algebra, and
there is a homotopy-invariant functor, $\mathbb E$, to the category
of spectra, defined by writing
$${\mathbb E}(X) = \KH (A\overline{X})$$

By theorem \ref{DL}, there is an associated
$G$-excisive functor ${\mathcal E}'$, and an assembly map
$$\alpha \colon {\mathcal E}'(X)\rightarrow \KH (A{\overline X})$$
such that the map
$$\alpha \colon {\mathbb E}' (G/H)\rightarrow {\mathbb E}(G/H)$$
is a stable equivalence whenever $H$ is finite.

\begin{definition}
The composition of the above map $\beta$ with the map $i_\ast \colon
\KH (A\overline{X})\rightarrow \KH (AG)$ induced by the faithful
functor $i\colon \overline{X}\rightarrow G$ is called the {\em
$KH$-assembly map} for the group $G$ over the ring $R$ with {\em
coefficients} in the $G$-algebra $A$.

We say that the group $G$ satisfies the {\em $KH$-isomorphism
conjecture} over $R$ with {\em coefficients} in the $G$-algebra $A$
if the assembly map
$$\beta \colon {\mathbb E}(\EG )\rightarrow \K (AG)$$
is a stable equivalence.
\end{definition}

The $KH$-assembly map is a variant of the Farrell-Jones assembly
map.  It was first defined and examined in \cite{BL}, where in
particular its relationship to the Farrell-Jones assembly map in
algebraic $K$-theory is analysed.\footnote{Actually, the
Farrell-Jones conjecture and the corresponding $KH$-isomorphism
conjecture are usually formulated in terms of virtually cyclic
groups rather than finite groups.  But by remark 7.4 in \cite{BL},
the above formulation of the $KH$-isomorphism conjecture is
equivalent to the original.}

The following result can be deduced directly from the above
definition.

\begin{theorem}
Consider the $\Or (G)$-spectrum
$${\mathbb E}(G/H) =\KH (G/H) = \KH (A\overline{G/H})$$
and let ${\mathbb E}'$ be the associated excisive functor.

Let $X$ be a path-connected space, and let $c\colon X\rightarrow +$
be the constant map.  Then up to stable equivalence the induced map
$$c_\ast \colon {\mathbb E}'(X)\rightarrow {\mathbb E}'(+)$$
is the $KH$-assembly map. \noproof
\end{theorem}

\section{Algebraic $KK$-theory and homology}

Let $G$ be a discrete group, and let $X$ be a right $G$-simplicial
complex (or just a {\em $G$-complex} for short).  Let us term $X$
{\em $G$-compact} if the quotient $X/G$ is a finite simplicial
complex.  Any $G$-complex is a direct limit of its $G$-compact
subcomplexes.

Given a ring $R$, the right action of $G$ on the space $X$ induces a
left-action of $G$ on the simplicial algebroid $R^X$.

\begin{definition}
Let $A$ be a $G$-algebra, and let $X$ be a $G$-complex.  Then we
define the {\em equivariant algebraic $K$-homology spectrum} of $X$
with {\em coefficients} in $A$ to be the direct limit
$$\K_{\mathrm{hom}}^G (X;A) = \lim_{{K\subseteq X} \atop {G-\mathrm{compact}}} \KK_G (R^K,A)$$

The associated {\em equivariant algebraic $K$-homology groups} are
defined by the formula
$$K_n^G (X;A) = \pi_n \K_{\mathrm{hom}} (X;A)$$
\end{definition}

\begin{theorem} \label{KhomA}
The functor $X\mapsto \K_{\mathrm{hom}}^G (X;A)$ is
$G$-homotopy-invariant and excisive.  For the one-point space, $+$,
we have $\K_\mathrm{hom}^G (+;A) = \KK_G (R,A)$.
\end{theorem}

\begin{proof}
Suppose that $X$ and $Y$ are $G$-compact simplicial $G$-complexes.
Let $f,g\colon X\rightarrow Y$ be equivariant simplicial maps.
Suppose that there is an elementary equivariant simplicial homotopy,
$F\colon X\times \Delta^1 \rightarrow Y$, between $f$ and $g$.  Then
we have an induced equivariant homomorphism $F^\ast \colon
R^Y\rightarrow R^X\otimes_R R^{\Delta^1} = R^X\otimes_\Z \Z [t]$ such
that $e_0 F^\ast = f^\ast$ and $e_1 F^\ast = g^\ast$.

By construction of equivariant $KK$-theory, algebraically
$G$-homotopic maps $\alpha , \beta \colon B\rightarrow B'$ induce
homotopic maps $\alpha^\ast ,\beta^\ast \colon \KK_G
(B',A)\rightarrow \KK_G (B,A)$.  It follows that the induced maps
$f_\ast ,g_\ast \colon \KK_G (X;A)\rightarrow \KK_G (Y;A)$ are
homotopic.

More generally, two $G$-homotopic equivariant simplicial maps
$f,g\colon X\rightarrow Y$ can be linked by a finite chain of
elementary homotopies, and the above argument again shows that the
induced maps $f_\ast ,g_\ast \colon \KK_G (X;A)\rightarrow \KK_G
(Y;A)$ are homotopic.  Generalising to non-compact $G$-complexes and
taking direct limits, it follows that the functor $X\mapsto
\K_\mathrm{hom}^G (X;A)$ is $G$-homotopy-invariant.

Let $X$ and $C$ be $G$-compact simplicial $G$-complexes, and suppose
we have a subcomplex $B\subseteq X$ and an equivariant map $f\colon
B\rightarrow C$.  Let $i\colon B\hookrightarrow X$ be the inclusion
map, and let $F\colon X\rightarrow X\cup_B C$ and $I\colon
C\rightarrow X\cup_B C$ be the maps associated to the push-out
$X\cup_B C$.  Then by functoriality, we have an induced pullback
diagram
$$\begin{array}{ccc}
R^C & \rightarrow & R^B \\
\uparrow & & \uparrow \\
R^{X\cup_B C} & \rightarrow & R^X \\
\end{array}$$

It is easy to check that we have a short exact sequence
$$0\rightarrow R^{X\cup_B C} \stackrel{(F^\ast , I^\ast)}{\rightarrow} R^X \oplus R^C \stackrel{i^\ast - f^\ast}{\rightarrow}R^B \rightarrow 0$$

The map $i\colon B\rightarrow X$ is an inclusion of a subcomplex.
It follows that we have an induced inclusion $i_\ast \colon
R^B\rightarrow R^X$.  This induced inclusion is an equivariant map,
but not an equivariant homomorphism.  Hence the above short exact
sequence has an $F$-splitting $(i_\ast , 0)$.

Therefore, by theorem \ref{les2}, we have natural maps $\partial
\colon K_n^G (X\cup_B C;A) \rightarrow K_{n-1}^G (B;A)$ such that we
have a long exact sequence
$$\rightarrow K_n^G(B;A)\stackrel{(i_\ast , -f_\ast )}{\rightarrow}K_n^G (X;A)\oplus K_n (C;A) \stackrel{F_\ast + I_\ast}{\rightarrow} K_n^G (X\cup_B C;A ) \stackrel{\partial}{\rightarrow} K_{n-1}^G(A)\rightarrow$$

Similar exact sequences can be seen to exist in the non-compact case
by taking direct limits.

To check the final axiom required of a $G$-homology theory, let $\{
X_i \ |\ i\in I \}$ be a family of $G$-compact simplicial
$G$-complexes.  Let $j_i \colon X_i \rightarrow \amalg_{i\in I} X_i$
be the canonical inclusion.  Then by definition of equivariant
$K$-homology as a direct limit, the map
$$\oplus_{i\in I} (j_i)_\ast \colon \oplus_{i\in I} K_n^G (X_i;A)\rightarrow K_n^G  (( \amalg_{i\in I} X_i );A)$$
is an isomorphism for all $n\in \Z$.

By definition, for the one-point space, $+$, we have
$\K_{\mathrm{hom}}^G (+;A) = \KK_G (R,A)$.
\end{proof}

\section{The Index Map}

In \cite{Mitch6}, the author used natural constructions of analytic
$KK$-theory spectra to prove that the Baum-Connes assembly map fits
into to the general assembly map machinery.  The $KH$-assembly map
is defined using the general machinery; the constructions in this
section prove that the map can be described using algebraic
$KK$-theory.  The methods are somewhat similar to those of
\cite{Mitch6}.

Let $G$ be a discrete group, and let $R$ be a ring.  Consider a
$G$-algebra $A$, over $R$, and a $G$-compact complex $K$.  According
to theorem \ref{descentealg}, we have a natural map of spectra
$$D\colon \KK_G (R^K,A)\rightarrow \KK (R^K G ,A G)$$

The algebra $R^K G$ is itself a finitely generated projective $R^K
G$-module; let us label this module ${\mathcal E}_K$.  Then by
theorem \ref{Mproducta} we have an induced morphism
$${\mathcal E}_K\wedge \colon \KK(R^K G, A G) \rightarrow \KH (A G)$$

Let $X = \EG$.  Then $X$ can be viewed as a $G$-simplicial complex.
Combining the above two maps and taking the
direct limit over $G$-compact subcomplexes, we obtain a map
$$\beta \colon \K_\mathrm{hom}^G(X;A)\rightarrow \KH (AG)$$

We call this map the {\em index map}, by analogy with the map
appearing in the Baum-Connes conjecture (see
\cite{BCH}).

Observe that by definition of the space $\EG$, we can take the
space $\EG$ to be a single point if the group $G$ is finite.
In this case, the above map is simply the composition
$$\KK_G (R,A)\rightarrow \KK (RG,AG) \rightarrow \KK (R,AG) \cong \KH(AG)$$

Now, for a ring $R$, let $M_\infty (R)$ be the $R$-algebra of all
infinite matrices over $R$, indexed by $\N$, where almost all
entries are zero.  Given an $R$-algebra $A$, let $\kappa \colon
A\rightarrow A\otimes_R M_\infty (R)$ be the homomorphism defined by
the formula $\kappa (a) = a\otimes P$, where $P\in M_\infty (R)$ is
the infinite matrix with $1$ at the top left corner, and every other
entry zero.  By construction of algebraic $KK$-theory, given
$R$-algebras
$A$ and $B$ the induced map
$$\kappa_\ast \colon KK_p (A,B) \rightarrow KK_p (A,B\otimes_R M_\infty (R))$$
is an isomorphism.

Let $G$ be a finite group.  Let $M_G (R)$ be the $R$-algebra of
matrices with entries in $R$ indexed by the group $G$.  Then there
is certainly a natural isomorphism
$$M_\infty (R) \cong M_\infty (R)\otimes_R M_G (R)$$

It follows that, given $G$-algebras $A$ and $B$, we have a natural
map $\kappa \colon B\rightarrow B\otimes_R M_G (R)$ inducing
isomorphisms
$$\kappa_\ast \colon KK_p^G (A,B) \rightarrow KK_p^G (A,B\otimes_R M_G (R))$$
at the level of $KK$-theory groups.

\begin{lemma}
Let $G$ be a finite group, and let $A$ be a $G$-algebra over a ring
$R$.  Then there is a natural injective homomorphism $\sigma \colon
AG \rightarrow A\otimes_R M_G (R)$, where the image is the set of
all elements of the $R$-algebra $A\otimes_R M_G (R)$ that are fixed
by the group $G$.
\end{lemma}

\begin{proof}
Let $V_G$ be the $R$-module consisting of all maps from $G$ to the
ring $R$.  Then $\End (V_G) = M_G (R)$.  Note that elements of the
tensor product $V_G\otimes_R A$ can be viewed as maps $s\colon
G\rightarrow A$.  Define homomorphism
$$\rho \colon G\rightarrow \End (V_G \otimes_R A) \qquad \pi \colon A \rightarrow \End (V_G \otimes_R A)$$
by the formulae
$$\rho (g)s(g_1) = s(gg_1) \qquad \pi (a) s(g_1) = \pi (g_1 (a)) s(g_1)$$
respectively, where $s\colon G\rightarrow A$ is an element of the
module $V_G\otimes_R A$, $g,g_1\in G$, and $a\in A$.

Then we have an associated injective homomrphism $\sigma \colon
AG\rightarrow \End (V_G\otimes_R A)$ defined by the formula
$$\sigma \left( \sum_{i=1}^m a_i g_i \right) = \sum_{i=1}^m \pi (a_i) \rho (g_i)$$

Now $\End (V_G\otimes_R A) = A\otimes_R M_G (R)$.  It is easy to
check that the image is the fixed point set.
\end{proof}

The proof of the following result is now adapted from the
corresponding result on the Baum-Connes conjecture, where it is
sometimes called the {\em Green-Julg theorem}.  See for example
\cite{GHT} for an elementary account.

\begin{theorem} \label{FINI}
The index map is a stable equivalence for finite groups.
\end{theorem}

\begin{proof}
Let $G$ be a finite group, and let $A$ be a $G$-algebra over the
ring $R$.  As we remarked above, we have a natural isomorphism
$\kappa_\ast \colon KK^G_p (R,A)\rightarrow KK^G_p (R,A\otimes_R
M_G(R))$.  Let $\sigma_\ast \colon \KK^G_p (R,AG)\rightarrow KK^G_p
(R,(A\otimes_R M_G(R))$ be induced by the isomorphism in the above
lemma.

We have an obvious canonical map
$$KH_p (AG) = KK_p (R,AG)\rightarrow KK_p^G (R,AG)$$
and we can define a homomorphism $\gamma \colon KH_p
(AG)\rightarrow
KK_p^G(R,A)$ fitting into a commutative diagram
$$\begin{array}{ccc}
KK^G_p (R,AG) & \stackrel{\sigma_\ast}{\rightarrow} & KK^G_p (R,A\otimes_R M_G(R)) \\
\uparrow & & \uparrow \\
KH_p (AG) & \stackrel{\gamma}{\rightarrow} & KK_G^p (R,A) \\
\end{array}$$

We claim that $\gamma$ is an inverse to the map $\beta \colon \KK_G
(R,A)\rightarrow \KK (R,AG)$ at the level of groups, thus proving
the result.  Looking at suspensions, it suffices to prove the case
when $p=0$.

Note that, by construction of the algebraic $KK$-theory groups, an
arbitary $KK$-theory class $[\alpha]\in KK_0^G (B,C)$ is represented
by a homomorphism $\alpha \colon J^{2n}B \rightarrow
C^{S^n}\otimes_R M_\infty (R)$.

Hence, consider a homomorphism $\alpha \colon J^{2n}R\rightarrow
A^{S^n}\otimes_R M_\infty (R)$.  Then we have a commutative diagram
$$\begin{array}{ccccc}
J^{2n}R & \rightarrow & J^{2n}RG & \stackrel{\alpha_\ast}{\rightarrow} & (A^{S^n}\otimes_R M_\infty (R))G \\
& & \downarrow & & \downarrow \\
& & J^{2n}R\otimes_R M_G(R) & \stackrel{\alpha_\ast}{\rightarrow} & A^{S^n}\otimes_R M_\infty (R) \\
\end{array}$$
where the vertical maps are versions of the map $\sigma$ in the
above lemma.

Now, at the level of $KK$-theory groups, the class $\beta \circ
\gamma [\alpha ]$ is the composition of the top maps in the above
diagram and the vertical map on the right.

On the other hand, the composition of the first map on the top and
the first map on the left is simply the stabilisation map $\kappa$.
Hence $[\alpha ] = [\alpha \otimes 1] = \beta \circ \gamma [\alpha
]$.  We have shown that $\beta \circ \gamma = 1_{KH (AG)}$.

Conversely, consider an equivariant homomorphism $\alpha \colon
J^{2n}R \rightarrow (AG)^{S^n}\otimes_R M_\infty (R)$.  Then the
composition with the map $\sigma$ defines a class in the $KK$-theory
group $KK_G (R,B)$.

The image $\beta [\alpha]$ is defined by the composition
$$J^{2n}R \rightarrow J^{2n}(RG) \stackrel{\alpha_\ast}{\rightarrow} ((AR)^{S^n}G)G$$

We can form the diagram
$$\begin{array}{ccccc}
J^{2n}RG & \stackrel{\alpha_\ast}{\rightarrow} & ((AR)^{S^n}G)G & \stackrel{\sigma_\ast}{\rightarrow} & (A\otimes M_R(G))^{S^n}G \\
\uparrow & & \uparrow & & \uparrow \\
J^{2n}R & \stackrel{\alpha}{\rightarrow} & (AR)^{S^n}G & = & (AR)^{S^n}G \\
\end{array}$$

The left square of the diagram commutes.  The right square of the
diagram does not commute.  The issue is that the middle vertical map
gives us a copy of the group $G$ on the right of the relevant
expession, whereas the vertical map on the right gives us a copy of
the ring $M_R(G)$ in the centre of the expression.

If we apply the homomorphsim $\sigma \colon (A\otimes_R M_G(R))G
\rightarrow A\otimes_R M_G(R)\otimes_R M_G(R)$, we see that the
square on the right commutes modulo the isomorphism $s\colon
M_G(R)\otimes_R M_G(R) \rightarrow M_G (R)\otimes_R M_G(R)$ defined
by writing $s(x\otimes y) = y\otimes x$.

The map $s$ is clearly naturally isomorphic to the identity map.  It
follows by lemma \ref{nisog} that the right square in the above
diagram commutates at the level of $KK$-theory groups.

Now the composite of the map on the left and the maps on the top row
of our diagram are the composition give us the class $\gamma \circ
\beta [\alpha ]$.  Thus $\gamma \circ \beta [\alpha ] = [\alpha ]$.

We see that the homomorphism $\gamma$ is the inverse of the
homomorphism $\beta$, and we are done.
\end{proof}

Given an equivariant map of $G$-complexes $f\colon X\rightarrow Y$,
there is an induced functor $f_\ast \colon \overline{X}\rightarrow
\overline{Y}$ between the transport groupoids.  There is an obvious faithful functor $i\colon
\overline{X}\rightarrow G$.  If $A$ is a $G$-algebra, it can
therefore also be regarded as an $\overline{X}$-algebra.

Now, let $K$ be a $G$-compact subcomplex.  Then we have an induced
restriction map
$$i^\ast \colon \KK_G (R^K,A)\rightarrow \KK_{\overline{X}}(R^K,A)$$

By theorem \ref{descentealg}, there is a natural map
$$D\colon \KK_{\overline{X}} (R^K,A)\rightarrow \KK(R^K\overline{X} ,A\overline{X})$$

\begin{definition}
Let $x\in X$.  Then we write ${\mathcal E}_K (x)$ to denote the set
of collections
$$\{ \eta_y \in \Hom (x,y)_{R^K\overline{X}} \ |\ y\in X \}$$
such that the formula
$$\eta_y g = \eta_z$$
holds for all elements $g\in G$ such that $yg=z$.
\end{definition}

The assignment $x\mapsto {\mathcal E}_K (x)$ is an
$R^K\overline{X}$-module, which we again label ${\mathcal E}_K$.  The
$R^K\overline{X}$-action is defined by composition of morphisms.

Since the action of the group $G$ on the space $X$ is not
necessarily transitive, the module $R^K$ is not necesarily finitely
generated and projective.  However, the restriction, ${\mathcal
E}|_{\Or (x)}$, to the path-component $(R^K \overline{X})_{\Or (x)}$
is isomorphic to the module $\Hom (-,x)_{R^K \overline{X}|_{\Or
(x)}}$.  Thus, the module ${\mathcal E}_K$ is almost
finitely-generated and projective, and we have a canonical morphism
$${\mathcal E}_K \wedge \colon \KK(R^K \overline{X}, A\overline{X}) \rightarrow  \KH (A\overline{X})$$

We can compose with the map
$$D\colon \KK_{\overline{X}} (R^K,A)\rightarrow \KK(R^K \overline{X} ,A\overline{X})$$
to form the composite
$$\gamma_K \colon \KK_G (R^K,A)\rightarrow \K (A\overline{X})$$
and take direct limits to obtain a map
$$\gamma \colon \K_\mathrm{hom}^G (X;A)\rightarrow \K (A\overline{X})$$

\begin{proposition}
Let $f\colon (X,K)\rightarrow (Y,L)$ be a map of pairs of
$G$-simplicial complexes, where the subcomplexes $K$ and $L$ are $G$-compact.  Let
$$f^\ast \colon R^L \overline{X}\rightarrow R^K \overline{X}
\qquad f_\ast \colon R^L \overline{X}\rightarrow R^L\overline{Y}$$
be the obvious induced maps.  Then there is an almost finitely
generated projective $R^L
\overline{X}$-module, $\theta$ ,such that $f_\ast \theta  =
{\mathcal E}_L$ and $f^\ast \theta = {\mathcal E}_K$.
\end{proposition}

\begin{proof}
Let $x\in X$.  Let $\theta (x)$ denote the set of collections
$$\{ \eta_y \in \Hom (x,y)_{R^L \overline{X}} \ |\ y\in X \}$$
such that the formula
$$\eta_y g = \eta_z$$
is satisfied for all elements $g\in G$ such that $yg=z$.

The assignment $x\mapsto \theta (x)$ forms an $R^L
\overline{X}$-module, $\theta$.  The $C_0(K)
\overline{X}$-action is defined by composition of morphisms.

The equations $f_\ast (\theta ) = {\mathcal E}_L$ and $f^\ast \theta
= {\mathcal E}_K$ are easily checked.
\end{proof}

\begin{lemma}
The map $\gamma \colon \K_\mathrm{hom}^G (X;A)\rightarrow \K
(A\overline{X})$ is natural for proper $G$-complexes.
\end{lemma}

\begin{proof}
Let $f\colon (X,K)\rightarrow (Y,L)$ be a map of pairs of
$G$-simplicial complexes, where the subcomplexes $K$ and $L$ are $G$-compact.  Then
by the above proposition, and naturality of the product and descent
map in algebraic $KK$-theory, we have a commutative diagram.
$$\xymatrix{
& \KK_{\overline{X}} (R^K,A) \ar[r] \ar[d] & \KK (R^K \overline{X} , A\overline{X} ) \ar[d] \ar[rd] & \\
\KK_G (R^K,A) \ar[ru] \ar[d] & \KK_{\overline{X}} (R^L,A) \ar[r] & \KK (R^L\overline{X}. A\overline{X} ) \ar[d] \ar[r] & \KH (A\overline{X} ) \ar[d] \\
\KK_G (R^L,A) \ar[ru] \ar[r] & \KK_{\overline{Y}} (R^L,A) \ar[u] \ar[rd] & \KK (R^L\overline{X}. A\overline{Y} ) \ar[r] & \KH (A\overline{Y} ) \\
& & \KK (R^L\overline{Y} , A\overline{Y} ) \ar[u] \ar[ru] & \\
}$$

Taking direct limits, the desired result follows.
\end{proof}

\begin{lemma}
The composite map $i_\ast \gamma \colon \K_\mathrm{hom}^G
(X;A)\rightarrow \KH (AG)$ is the index map.
\end{lemma}

\begin{proof}
Let $K$ be a $G$-compact subcomplex of $X$.  The naturality properties
of the various descent maps and products give us a commutative
diagram
$$\xymatrix{
\KK_G (R^K,A) \ar[r] \ar[dd] &\KK (R^K G, AG) \ar[r] \ar[d] & \KH (AG) \\
& \KK(R^K \overline{X}, AG) \ar[ru] & \\
\KK_{\overline{X}} (R^K,A) \ar[r] &\KK (R^K\overline{X}, A\overline{X}) \ar[r] \ar[u] & \KH (A\overline{X}) \ar[uu] \\
}$$

Here the top row is the index map, and the composite of the bottom
row and the vertical map on the left is the map $\gamma$. Taking
direct limits, the desired result follows.
\end{proof}

\begin{lemma}
Let $H$ be a finite subgroup of $G$.  Then the map $\gamma
\colon \K_\mathrm{hom}^G (G/H;A)\rightarrow \KH (A\overline{G/H})$
is a stable equivalence of spectra.
\end{lemma}

\begin{proof}
Let $i\colon H\hookrightarrow G$ be the inclusion isomorphism.  Then
the group $H$ and groupoid $\overline{G/H}$ are equivalent.  By
theorem \ref{reseq}, and the naturality of restriction maps, the map
$\gamma$ is equivalent to the composition
$$\KK_G(R^{G/H},A)\stackrel{i^\ast}{\rightarrow} \KK_H (R^{G/H},A) \stackrel{D}{\rightarrow} \KK (R^{G/H} H, AH ) \stackrel{{\mathcal E}_{G/H}\wedge}{\rightarrow} \KH (AH)$$

Let $j\colon R \rightarrow R^{G/H}$ be induced by the constant map $G/H\rightarrow  +$.
Then we have a commutative diagram
$$\begin{array}{ccc}
\KK_G (R^{G/H},A) & & \\
\downarrow & & \\
\KK_H (R^{G/H},A) & \stackrel{j^\ast}{\rightarrow} & \KK_H (R,A) \\
\downarrow & & \downarrow \\
\KK_H (R^{G/H}H,AH) & \stackrel{j^\ast}{\rightarrow} & \KK_H (RH,AH) \\
\downarrow & & \downarrow \\
\K (AH) & = & \K (AH) \\
\end{array}$$

A straightforward calculation tells us that the composite $j^\ast
i^\ast \colon \KK_G (R^{G/H},A)\rightarrow \KK_H (R,A)$ is a stable
equivalence of spectra.  The composite map on the right,
$\beta
\colon \KK_H (R,A)\rightarrow \K (AH)$ is the index map. Since the
group $H$ is finite, the space $+$ is a model for the classifying
space
$E(H,\EG )$.

But the index map is a stable equivalence for finite
groups by theorem \ref{FINI}, so we are done.
\end{proof}

By theorem \ref{KhomA}. the equivariant $K$-homology functor
$\K_\mathrm{hom}^G(-;A)$ is $G$-excisive.  We can therefore use the
above three lemmas to apply theorem \ref{DL} to the study of the
index map; we immediately obtain the following result.

\begin{theorem}
Let ${\mathbb E}'$ be a $G$-excisive functor from the category of
proper $G$-$CW$-complexes to the category of symmetric spectra.
Suppose we have a natural transformation $\alpha \colon {\mathbb
E}'(X)\rightarrow \K (A\overline{X})$ such that the map
$$\alpha \colon {\mathbb E}'(G/H)\rightarrow \K (A\overline{G/H})$$
is a stable equivalence for every finite subgroup, $H$, of the group
$G$.

Then, up to stable equivalence, the composite $\alpha i_\ast \colon
{\mathbb E}' (X)\rightarrow \K (AG)$ is the map $\beta$. \noproof
\end{theorem}

By definition of the $KH$-assembly map, the following therefore
holds.

\begin{corollary}
The index map is the $KH$-assembly map. \noproof
\end{corollary}

Now, the $KH$-isomorphism conjecture holds for finite groups with
any coefficients.  It follows from more general results in \cite{BL}
that the conjecture also holds for the integers.

\begin{theorem}
Let $G$ be a finitely generated abelian group.  Then the
$KH$-isomorphism conjecture holds for $G$ with any coefficients.
\end{theorem}

\begin{proof}
By the fundamental theorem of abelian groups, we have an isomorphism
$$G\cong \Z^q \oplus {\mathbb Z}/p_1 \oplus \cdots \oplus {\mathbb Z}/p_k$$
where the $p_i$ are prime numbers.  We have classifying spaces for
proper actions
$$\underline{E}\Z  = \R \qquad \underline{E}\Z/p_i = + .$$

Thus the group $G$ has
classifying space ${\mathbb R}^q$.  The copies of the group $\Z$ act
by translation, and the copies of finite groups act trivially.

It follows that we have a commutative diagram
$$\begin{array}{ccc}
K_n^G (\EG ;A) & \cong & K_n^\Z ( \R ;A)^q \oplus K_n^{\Z/p_1} (+;A) \oplus \cdots \oplus K_n^{\Z/p_k} (+;A) \\
\downarrow & & \downarrow \\
KH_n^G (AG) & \cong & KH_n (A\Z )^q \oplus KH_n( A\Z/p_1) \oplus \cdots \oplus K_n (A\Z/p_k) \\
\end{array}$$
where the vertical arrows are copy of the assembly map at the level
of groups.

We know that the $KH$-isomorphism conjecture holds for finite groups
and for the integers.  The direct sum of assembly maps on the right
is thus an isomorphism.

It follows that the map on the left is also an isomorphism, and we
are done.
\end{proof}

\bibliographystyle{plain}

\end{document}